\documentclass[a4paper,oneside,onecolumn]{article}
\usepackage{appendix}
\usepackage{multirow}
\usepackage[T1]{fontenc} 

\usepackage[utf8]{inputenc} 

\usepackage{graphics,graphicx,epstopdf} 

\usepackage{enumitem} 

\usepackage{subfig} 

\usepackage{amsmath,amssymb,amsthm,amsfonts} 

\newcommand{\vertiii}[1]{{\left\vert\kern-0.25ex\left\vert\kern-0.25ex\left\vert #1 
    \right\vert\kern-0.25ex\right\vert\kern-0.25ex\right\vert}}
\usepackage{dsfont}

\usepackage{hyperref} 

\usepackage{color, xcolor}
\usepackage{float}
\usepackage{authblk} 

\usepackage[nameinlink]{cleveref}

\theoremstyle{plain} 
\newtheorem{theorem}{Theorem}[section]
\newtheorem{lemma}{Lemma}[section]

\theoremstyle{definition} 

\newtheorem{example}{Example}[section]
\theoremstyle{remark} 
\newtheorem{remark}{Remark}[section]

\crefname{section}{Section}{Sections}
\crefname{subsection}{subsection}{subsections}
\crefname{theorem}{Theorem}{Theorems}
\crefname{lemma}{Lemma}{Lemmas}
\crefname{proposition}{Proposition}{Propositions}
\crefname{corollary}{Corollary}{Corollarys}
\crefname{definition}{Definition}{Definitions}
\crefname{example}{Example}{Examples}
\crefname{remark}{Remark}{Remarks}
\crefname{table}{Table}{Tables}
\Crefname{figure}{Figure}{Figures}
\Crefname{equation}{}{}
\definecolor{c1}{rgb}{0,0,1} 
\definecolor{c2}{rgb}{0,0.3,0.9} 
\definecolor{c3}{rgb}{0.3,0,0.9} 

\hypersetup{
	pdfauthor={Aditi Tomar},
	pdfsubject={},
	pdftitle={},
	pdfkeywords={},
	breaklinks = true, 
	linktocpage=true, 
	colorlinks=true,       
	linkcolor={c1},          
	citecolor={c2},        
	filecolor=black,      
	urlcolor={c3},       
}

\numberwithin{equation}{section}

\allowdisplaybreaks
\pdfoutput=1
\usepackage{cancel}
 
\usepackage{soul}
\usepackage[left=2.0cm,top=2.5cm,right=2.0cm,bottom=2.5cm,bindingoffset=0.0cm]{geometry}
\definecolor{green}{HTML}{027148}
\begin{document}
\title{\textbf{Optimal error estimates of a non-uniform IMEX-L1 finite element method for time fractional PDEs and PIDEs}}

\author{\textsc{Aditi Tomar}\thanks{\href{mailto: aditi183321002@iitgoa.ac.in}{aditi183321002@iitgoa.ac.in}} }
\author{\textsc{Lok Pati Tripathi }\thanks{\href{mailto: lokpati@iitgoa.ac.in}{lokpati@iitgoa.ac.in}}}
\affil{School of Mathematics and Computer Science, Indian Institute of Technology Goa, Goa-403401, India.}
\author{\textsc{Amiya K. Pani }\thanks{\href{mailto: amiyap@goa.bits-pilani.ac.in}{amiyap@goa.bits-pilani.ac.in}}}
\affil{Department of Mathematics, BITS-Pilani, KK Birla Goa Campus, Goa-403726, India.}

\date{} 

\maketitle 

\begin{abstract}
Stability and optimal convergence analysis of a non-uniform implicit-explicit L1 finite element method (IMEX-L1-FEM) is studied for a class of time-fractional linear partial differential/integro-differential equations with non-self-adjoint elliptic part having (space-time) variable coefficients. The proposed scheme is based on a combination of an IMEX-L1 method on graded mesh in the temporal direction and a finite element method in the spatial direction. With the help of a discrete fractional Gr\"{o}nwall inequality, global almost optimal error estimates in $L^2$- and $H^1$-norms are derived for the problem with initial data $u_0 \in H_0^1(\Omega)\cap H^2(\Omega)$. The novelty of our approach is based on managing the interaction of the L1 approximation of the fractional derivative and the time discrete elliptic operator to derive the optimal estimate in $H^1$-norm directly. Furthermore, a super convergence result is established when the elliptic operator is self-adjoint with time and space varying coefficients, and as a consequence, an $L^\infty$ error estimate is obtained for 2D problems that too with the initial condition is in $  H_0^1(\Omega)\cap H^2(\Omega)$. All results proved in this paper are valid uniformly as $\alpha\longrightarrow 1^{-}$, where $\alpha$ is the order of the Caputo fractional derivative. Numerical experiments are presented to validate our theoretical findings.
\end{abstract}
 \section{Introduction}\label{section1}
This paper is devoted to a non-uniform implicit-explicit L1 finite element method (IMEX-L1-FEM) for the following class of time-fractional linear partial differential / integro-differential equations (PDEs/PIDEs):

\begin{align}\label{pide}
\begin{cases}
\partial^{\alpha}_{t} u  +  \mathcal{L}u - \lambda\mathcal{I}u \;=\; f & \quad \text{in} \quad \Omega\times J,\\[5pt]
u ~= ~0 &\quad \text{on}\quad \partial\Omega\times J, \\[5pt]
u(\boldsymbol{x},\;0) ~= ~u_0(\boldsymbol{x}) &\quad \forall \boldsymbol{x}\in \Omega,
\end{cases}
\end{align}
where $\partial^{\alpha}_{t} u$ is the Caputo fractional derivative of order $\alpha\; (0<\alpha < 1)$ with respect to $t$,
\begin{align}
    \label{caputoderivative} \partial^{\alpha}_{t} u:&=\;\int_{0}^{t}k_{1-\alpha}(t-s)\partial_{s}u\;ds\quad \text{with kernel }\; k_{\beta}(t):=\;\frac{t^{\beta-1}}{\Gamma{(\beta)}},\;t>0,\; \beta>0,\\
    \label{EllipticOperator} \mathcal{L}u:&=\;  -\nabla\cdot(\boldsymbol{A}\nabla u)  + \boldsymbol{b}\cdot\nabla u + cu \quad \text{and} \quad \mathcal{I}u(\boldsymbol{x},t):=\; \int_{\Omega}u(\boldsymbol{y},t)g(\boldsymbol{x},\boldsymbol{y})\;d\boldsymbol{y}.
\end{align}
Here, $\lambda\in \mathbb{R},$ $g\in L^{\infty}(\Omega\times\Omega),$  $\Gamma(\cdot)$ denotes the Gamma function, $J:=(0,\;T],\;0<T<\infty,$ and $\Omega\subset\mathbb{R}^{d}~(d=1,2,3)$ is a convex polygonal or polyhedral bounded domain. Appropriate regularity requirements on the functions $\boldsymbol{A}:\Omega\times J\to \mathbb{R}^{d\times d}$, $\boldsymbol{b}:\Omega\times J\to \mathbb{R}^{d},$ $c:\Omega\times J\to \mathbb{R},$ $f:\Omega\times J\to \mathbb{R}$ and $u_0$ are mentioned in Section~\ref{section2}. For a.e. $(\boldsymbol{x},\;t)\in\Omega\times J$, $\boldsymbol{A}(\boldsymbol{x},\;t)$ is a real symmetric and uniformly positive definite matrix in the sense that there exists a positive constant $\kappa_{0}$ and $\kappa_{1}$ such that
\begin{align}
    \label{supdc} 
    &\kappa_{0}\;\|\boldsymbol{\xi}\|^{2} \;\leq\; \boldsymbol{\xi}^{T}\boldsymbol{A}\boldsymbol{\xi}\; \leq\; \kappa_{1}\;\|\boldsymbol{\xi}\|^{2} \quad \forall~\boldsymbol{\xi}\in \mathbb{R}^{d}\quad \text{a.e. in }~ \Omega\times J.
\end{align}
Non-local integral operator $\mathcal{I}: L^{1}(\Omega)\to L^{\infty}(\Omega)$ with a particular type of kernel $g(\boldsymbol{x},\boldsymbol{y})=\rho(\boldsymbol{y}-\boldsymbol{x})$ appears in option pricing problems under Merton's and Kou's jump-diffusion models \cite{ MR2182141, MR3633783,kou2002jump, MR2873249, Merton1976}, where $\rho:\mathbb{R}^{d}\to\mathbb{R}^{+}$ is a probability density function (see Example \ref{Mertons}). For $\lambda = 0$, the problem (\ref{pide}) becomes a time-fractional linear general parabolic PDE with variable coefficients. Using the properties of the non-local integral operator $\mathcal{I}$ and following \cite{MR4290515, MR3814402}, the well-posedness of the problem (\ref{pide}) can be carried out for the case of $\lambda\neq 0$ (see, Section~\ref{section2}). The main computational advantage of IMEX methods over fully implicit methods is that, at each time level, we need to invert only a sparse matrix instead of a dense matrix \cite{MR3633783, MR2873249}. Fast Fourier Transform (FFT) can also be applied to compute the matrix-vector product resulting after the discretization of the integral term, see \cite{MR3633783, MR2873249}. 

In the last few decades, applications of fractional PDEs/PIDEs have gained wide popularity in various fields such as physics \cite{ MR1890104}, chemistry \cite{MR2025566}, biology \cite{ MR2595930}, finance \cite{ MR3585482}, etc. The analytical solutions to such model problems are rarely available. Therefore, we rely on its numerical approximations, see \cite{ MR3273144, MR3449907, MR3738850, MR3904430, MR4090355, MR3173143}. Since the time derivative of the solution $u$ of the problem (\ref{pide}) has a singularity at $t=0$, therefore, standard numerical methods break down and the problem needs special care for the singularity \cite{MR3463438, MR2832607, MR3589365}. Earlier, this fact was ignored and the truncation error of numerical methods for approximating fractional derivative is analyzed by assuming the boundedness of second or third-order time derivatives uniformly in $[0, T]$, see \cite{MR3148558, MR2349193, MR3143842}.   In general,
$\|\partial_t u(t)\|\leq C_ut^{\alpha-1}$ for $0<t \leq T$, where the constant $C_u >0$ may
depend on $T$ \cite{  MR2832607, MR3589365,MR3639581}. This shows that $\partial_t u$ blows up as $t \rightarrow 0^+$ and is bounded away from $t=0$, while the solution is continuous at $t=0$.  

The error analysis of the L1 scheme without ignoring the initial singularity has been carried out using a uniform temporal mesh with $\Delta t = T/N$ in \cite{ MR3894161, MR3957890, MR3744997}. Yan et al. \cite{MR3744997} have proposed a modified L1 scheme for the problem (\ref{pide}) with $\mathcal{L}=-\Delta,$ and $\lambda=0$ and obtained 
\begin{align*}
    \|u(t_n)-U^n\|\leq C t_n^{\alpha-2}\Delta t^{2-\alpha}\|u_0\|,
\end{align*}
where $U^n$ is the discrete-time approximation of $u(t_n)$.

Our main focus will be on graded temporal meshes as these meshes concentrate the grid points near $t= 0$ and are reasonably convenient for numerical approximations of solutions that are singular at $t= 0$ \cite{ MR3957889,  MR3790081, MR3904430, MR2813191,   MR3347460, MR4199354, MR3639581}. 
 More recently, Stynes et al. \cite{MR3639581} have analyzed the L1 formula on graded time grids of the form $t_n = (n/N)^{\gamma} T,~ 0 \leq n \leq N$ and $\Delta t_n = t_{n} -t_{n-1}$ combined with FDM in spatial direction for the problem (\ref{pide}),  where $\mathcal{L} u= - u_{xx}+cu$, $\gamma \geq 1$ is the grading parameter and $N$ is the number of grid points in the time direction. They have obtained an error estimate
by using the discrete maximum principle and direct analysis of local truncation error under the regularity assumption  
\begin{align*}
& \|\partial_t^k u(\cdot, t)\| \lesssim 1 + t^{\alpha-k}, \quad k\in\{0,1,2\}.
\end{align*}
 They have shown that given the typical singular behavior of $u(t)$ at $t=0$, the maximum error in the fully discrete solution is of order $N^{- \min\{2 - \alpha,\gamma \alpha \}}$. Therefore, when $\gamma =1$, which corresponds to the uniform time mesh, the error is of $O(N^{-\alpha} )$, and for $\gamma \geq (2 - \alpha )/\alpha$, the error is of $O(N^{-(2-\alpha)})$.

In \cite{MR3790081}, a sharp error estimate for
the L1 formula on nonuniform meshes was obtained for (\ref{pide}), where $\mathcal{L}u = -u_{xx} - \kappa u$ with $\kappa >0$ using FDM
based on a discrete fractional Grönwall inequality. In \cite{MR3904430}, Liao et al. have proposed fractional Gr\"{o}nwall inequality where $\mathcal{L}$ is a strongly elliptic linear operator in the spatial variable. In \cite{MR3957889}, Kopteva discussed the L1-type discretizations on
graded time meshes for a fractional-order linear parabolic equation
where the general second-order linear elliptic operator with only space-dependent coefficients. Mustapha, in \cite{MR3802434}, has studied a semidiscrete
Galerkin FEM for the time-fractional diffusion equations with time-space
dependent diffusivity coefficient, i.e., when
$\mathcal{L}u(x,t) = -\nabla \cdot (\textbf{A}(x, t) \nabla u(x, t))$
 and the optimal error bounds in $L^2$- and $H^1$-norms are obtained for the semi-discrete problem. Again, in \cite{MR3957890}, Jin et al. have obtained error analysis for the fully discrete solution of the subdiffusion equation with self-adjoint time-dependent elliptic part using Galerkin FEM with conforming piecewise linear finite elements in space and backward Euler convolution quadrature in time using uniform temporal mesh. Unlike the above literature, in our case, $\mathcal{L}$ is a general non-self-adjoint linear uniformly elliptic operator with time-dependent coefficients, which brings challenges in the analysis of the standard nonuniform approximations of (\ref{pide}). 
 
To the best of our knowledge, there is hardly any literature available on non-uniform (IMEX) L1 finite element method for a time-fractional PDE or PIDE equipped with a general elliptic operator of the form (\ref{EllipticOperator}) having time and space-dependent coefficients. In this article, an effort has been made to fill this gap by establishing the stability and optimal convergence analysis of a non-uniform implicit-explicit L1 finite element method for the problem (\ref{pide}). The main contributions of the present work are:
\begin{itemize}
\item to discuss new regularity results. Since it is a perturbation by lower order derivative terms of \cite{MR4290515}, the proofs are given in Appendix~\ref{appendix_sec_reg_result_a}.
    \item to establish $L^2$- and $H^1$-norms stability of an IMEX-L1-FEM on graded mesh  using a discrete Gr\"{o}nwall lemma for the problem (\ref{pide}), where $\mathcal{L}$ is non-self-adjoint elliptic linear operator having space and time-dependent coefficients. It is to be noted that for the $H^1$ estimate, one has to take special care due to time-dependent coefficients and their effect on the discrete fractional derivative.
     \item to derive a global almost optimal order error estimate $O\left((h^{2-m} + N^{-(2-\alpha)} )\log_e(N)\right),\;m=0,1$ (see, Theorem \ref{L2H1errortheorem}) with respect to $H^m$-norm, $H^0:=L^2$ for the problem (\ref{pide}) under the assumption $u_0 \in H_0^1(\Omega)\cap H^2(\Omega)$, where $N+1$ denotes the number of grid points in the temporal direction and $h$ is the maximum diameter of finite elements. 
     \item to establish an $L^{\infty}$-norm error estimate for 2D problems  (see, Theorem~\ref{Linftynormestimate}), as a result of superconvergence result for (\ref{pide}) with $u_0 \in H_0^1(\Omega)\cap H^2(\Omega)$. 
     \item inspired by \cite{MR4246866}, to ensure the validity of all estimates in this article, even as $\alpha$ approaches $1^-$.
\end{itemize}

The novelty of our approach is to exploit the interaction of the discrete in time and space elliptic operator with a discrete fractional derivative for deriving direct optimal convergence in $H^1$-norm. Moreover, using the discrete-time weights, superconvergence results are derived for the problem (\ref{pide}) with data in $H^2(\Omega)\cap H_0^1(\Omega)$ when the elliptic operator is self-adjoint with space and time-varying coefficients. To make the presentation simple, we have not considered semi-linear problems. However, without any significant change, the proposed method and its analysis can be extended for semi-linear time-fractional PDEs/PIDEs under certain regularity assumptions on the non-linear function.
Here, $H^{m}(\Omega)$ and $ H_{0}^{m}(\Omega),\;m\in\mathbb{N}\cup\{0\},$ are standard Sobolev spaces equipped with standard norm $\|\cdot\|_{m}$ (see, \cite{MR2424078}), $H_{0}^{0}(\Omega)\;=\; H^{0}(\Omega)\;=\; L^{2}(\Omega)$, and $\|\cdot\|_{0} = \|\cdot\|$ with $(\cdot,\cdot)$ denotes the $L^{2}$-inner product, $\langle\cdot,\cdot\rangle$ is the duality paring between $H_{0}^{1}(\Omega)$ and its topological dual $(H_{0}^{1}(\Omega))^{\ast}:=H^{-1}(\Omega),$ $\langle \phi,\psi\rangle = (\phi,\psi) \;\forall \phi\in L^{2}(\Omega)$ and $\forall \psi\in H_{0}^{1}(\Omega).$ For a given Hilbert space $\mathcal{H}$, $W^{m,p}(J;\mathcal{H}), m\in\mathbb{N}\cup\{0\}, 1\leq p \leq \infty$, $W^{0,2}(J;\mathcal{H}):=L^2(J;\mathcal{H})$ denotes the standard Bochner-Sobolev spaces.  

The rest of the paper is structured as follows. Section~\ref{section2} deals with the variational formulation of the problem (\ref{pide}) and provides some results that will be used in the subsequent analysis. Non-uniform IMEX-L1-FEM is proposed in section~\ref{section3}. Section~\ref{section4} is devoted to the stability analysis of the proposed method. In section~\ref{section5}, optimal error estimates are derived. In section~\ref{Linftyerroranalysis}, an $L^{\infty}$ error estimate is obtained for $d=2$. Numerical experiments are presented in section~\ref{section6} to validate our theoretical findings. The article is concluded in section~\ref{section8}.

Throughout this article, $C$ denotes a positive generic constant (not necessarily the same at each occurrence), which is independent of the approximation parameters, such as the maximum diameter $h$ of finite elements, the number of grid points $N+1$ in the temporal direction etc.

%
\section{Variational formulation, well-posedness, and some useful results}\label{section2}
The variational formulation of the problem (\ref{pide}) is to find $u :J\to H_{0}^{1}(\Omega)$ such that
\begin{align}
\label{variation1}
&\begin{cases}
& u(0) = u_0 ,\\
&\langle \partial_t^{\alpha} u,v \rangle +a(t;u,v) - \lambda (\mathcal{I}u,v ) ~=~ \langle f,v\rangle  \quad \forall v \in H_0^1(\Omega) \quad a.e. \quad t \in J,
\end{cases}
\end{align}
where the bilinear form $a(t;\cdot,\cdot):H_{0}^{1}\times H_{0}^{1}\to \mathbb{R}$ is defined as $\displaystyle a(t;u,v):=\langle \mathcal{L}u, v\rangle = (\boldsymbol{A}\nabla u, \nabla v) + (\boldsymbol{b}. \nabla u,v) +(cu,v)$, and the functions $\boldsymbol{A}:\Omega\times J\to \mathbb{R}^{d\times d}$, $\boldsymbol{b}:\Omega\times J\to \mathbb{R}^{d},$ $c:\Omega\times J\to \mathbb{R}$ satisfy the following boundedness properties:
\begin{align}\label{conditions}
    & \left\|\partial_t\boldsymbol{A}(\cdot,t)\right\|_{L^{\infty}(\Omega,\mathbb{R}^{d\times d})}+\|\boldsymbol{A}(\cdot,t)\|_{L^{\infty}(\Omega,\mathbb{R}^{d\times d})}+\|\boldsymbol{b}(\cdot,t)\|_{L^{\infty}(\Omega,\mathbb{R}^{d})} +\|c(\cdot,t)\|_{L^{\infty}(\Omega)}\;\leq\; C\quad  \forall t\in J,
\end{align}
for some positive constant $C$.
Define two bilinear forms $a_{0}(t;\cdot,\cdot):H_{0}^{1}\times H_{0}^{1}\to \mathbb{R}$ and $a_{1}(t;\cdot,\cdot):H_{0}^{1}\times H_{0}^{1}\to \mathbb{R}$ in such a way that
\begin{align}
    \label{bf} & a(t;\phi,\psi)\; = \;a_{0}(t;\phi,\psi) \;+\; a_{1}(t;\phi,\psi) \quad \forall \phi,\psi\in H_{0}^{1}(\Omega), \text{ a.e. } t\in J,
\end{align}
where
\begin{align}
\label{bf0} & a_{0}(t;\phi,\psi):= \; \big(\boldsymbol{A}(\cdot,t)\nabla\phi,\nabla\psi\big) \quad \forall \phi,\psi\in H_{0}^{1}(\Omega), \text{ a.e. } t\in J,\\   
    \label{bf1} &  a_{1}(t;\phi,\psi):= \big(\boldsymbol{b}(\cdot,t)\cdot\nabla\phi + c(\cdot,t)\phi,\psi\big) \quad \forall \phi,\psi\in H_{0}^{1}(\Omega), \text{ a.e. } t\in J.    
\end{align}
Note that the bilinear form $a_{0}(t;\cdot,\cdot):H_{0}^{1}\times H_{0}^{1}\to \mathbb{R}$ is symmetric, i.e, $a_0(t;\phi,\psi)=a_0(t; \psi,\phi) \quad \forall \phi,\psi\in H_{0}^{1}(\Omega), \text{ a.e. } t\in J$. Also, it satisfies the following boundedness, coercive, and Lipschitz continuous properties:
\begin{align}
\label{bounded}|a_{0}(t;\phi,\psi)|\;\leq \; \gamma_{0}\|\phi\|_{1}\|\psi\|_{1} \quad \forall \phi,\psi\in H_{0}^{1}(\Omega), \text{ a.e. } t\in J,\\
\label{coercive}a_{0}(t;\phi,\phi)\;\geq \; \beta_{0}\|\phi\|_{1}^{2} \quad \forall \phi\in H_{0}^{1}(\Omega), \text{ a.e. } t\in J,\\
\label{Lipschitz}|a_0(t;\phi,\phi)- a_0(s;\phi,\phi)|\leq L|t-s|\|\phi\|^2_1
\quad \forall \phi\in H_0^1(\Omega), \; \forall t,s\in J,
\end{align}
where $L$ is some Lipschitz constant. The bilinear form $a_{1}(t;\cdot,\cdot):H_{0}^{1}\times H_{0}^{1}\to \mathbb{R}$ is non-symmetric and bounded
\begin{align}
 \label{a1bounded1}|a_{1}(t;\phi,\psi)|\;\leq \; (\|b\|_{\infty}\|\nabla \phi\|+\|c\|_{\infty}\|\phi\|)\|\psi\|\;\leq \; \beta_{1}\|\phi\|_{1}\|\psi\| \quad \forall \phi,\psi\in H_{0}^{1}(\Omega), \text{ a.e. } t\in J,\\
\label{a1bounded2}|a_{1}(t;\phi,\psi)|\;\leq \; \beta_{1}\|\phi\|\|\psi\|_{1} \quad \forall \phi,\psi\in H_{0}^{1}(\Omega), \text{ a.e. } t\in J.  
\end{align}
From (\ref{a1bounded1}), we observe that if $\boldsymbol{b}=0$, then
\begin{align}\label{a1bounded1new}
 |a_1(t;\phi,\psi)| \leq \|c\|_{\infty}\|\phi\|\;\|\psi\|.  
\end{align}
Moreover, the bilinear form $a_{\lambda}(t;u,v):=a(t;u,v) - \lambda (\mathcal{I}u,v )$ associated with the problem (\ref{variation1}) satisfies the following boundedness and G\"{a}rding's inequality, respectively,
\begin{align*}
    & |a_{\lambda}(t;u,v)|\;\leq\; c_0 \|u\|_1\|v\|_1\quad\text{and}\quad a_{\lambda}(t;u,u)\;\geq \; c_1\|u\|^2_1 - c_2\|u\|^2
\end{align*}
for some constants $c_0,c_1>0$, and $c_2\geq 0$. Hence, for $u_0 \in L^2(\Omega)$ and $f \in L^2(J, H^{-1}(\Omega)),$ Theorem~$6.2$ in \cite{MR4290515} (see, Theorem 1.1 in \cite{MR3814402} when $\lambda=0$) ensures that the problem (\ref{variation1}) is well-posed, i.e., it has a unique solution $u\in H^{\alpha}(u_0; H_0^1(\Omega), L^{2}(\Omega))$ satisfying
\begin{align*}
    &\int_0^T\|\partial_t^{\alpha} u(t)\|^2_{-1}\;dt + \int_0^T\|u(t)\|^2_1\;dt\;\leq\; c \left(\|u_0\|^2 + \int_0^T \|f(t)\|^2_{-1}\;dt\right),
\end{align*}
where 
\begin{align*}
& H^{\alpha}\left(u_0;\; H_{0}^{1}(\Omega),\;L^{2}(\Omega)\right)\;:=\;\Big\{u\in L^{2}(J;H_{0}^{1}(\Omega)):\\
& \hspace{4.5cm} \int_{0}^{t}(t-s)^{-\alpha}(u(s)-u_0)ds\in \; _{0}H^{1}(J;H^{-1}(\Omega)) \Big\},
\end{align*}
and
\begin{align*}
    & {}_{0}H^{1}(J;H^{-1}(\Omega)):=\{ v \in L^2(J;H^{-1}(\Omega)), \;v_t \in L^2(J;H^{-1}(\Omega)), \text{ and $v$ vanishes}\\
    &\hspace{5cm} \text{at $t=0$ in the sense of trace} \}.
\end{align*}
Now, we provide the following regularity results for our subsequent use.
\begin{theorem}\label{Regularity condition} Let $f \in W^{3,1}(J;L^2(\Omega))$, $u_0 \in H_0^1(\Omega)\cap H^2(\Omega)$. Suppose for $\phi\in H^j_0(\Omega)$
\begin{align}\label{integralbounded}
  \|\mathcal{I}\phi\|_j\;\leq\;\beta_2\|\phi\|_j\quad \text{for}\; j=0,1, 
\end{align}

and for some positive constant $C_1$
\begin{align*}
        & \left\|\partial_t^k\boldsymbol{A}(\cdot,t)\right\|_{W^{1,\infty}(\Omega,\mathbb{R}^{d\times d})}+\left\|\partial_t^k\boldsymbol{b}(\cdot,t)\right\|_{L^{\infty}(\Omega,\mathbb{R}^{d})} +\left\|\partial_t^k c(\cdot,t)\right\|_{L^{\infty}(\Omega)}\;\leq\; C_1\; \forall t\in J\;\text{ and }\; k=0,1,2,3.
\end{align*}
Then, there exists a positive constant $C$ such that
\begin{align*} 
&u\in C^{2}((0,T]; H^2(\Omega)),\; \|u(t)\|_2 + t\|\partial_{t}u(t)\|_2 \leq C,\; t^k\|\partial_{t}^{k}u(t)\|_m\leq\; C \;t^{(2-m)\alpha/2},\; \;m=0,1,\; k = 1,2,\; \forall t\in J.
\end{align*}
\end{theorem}
\begin{proof}
   When $\boldsymbol{b}=0, \; g=0$, $u_0 \in H_0^1(\Omega)\cap H^2(\Omega)$, these regularity results have been derived in
(see, Theorem 6.15 \cite{ MR4290515}, Theorem 2.2 \cite{MR3957890}). Since the present problem is a perturbation by lower order terms and therefore its proof is given in Appendix~\ref{appendix_sec_reg_result_a}. Essentially, the result is proved using Theorem~\ref{reg_results_a1} and Theorem~\ref{reg_results_a1_L2} and interpolation.
\end{proof}
\begin{remark}
    Based on \cite{ MR4290515}, arguments given in the Appendix~\ref{appendix_sec_reg_result_a} can be suitably modified with more refined analysis to take care of nonsmooth data. We shall not pursue it here as our objective is to prove results for smooth data as shown in the above Theorem.
\end{remark}

%
\section{Non-uniform IMEX-L1-FEM}\label{section3}
%
Let $\{t_n = \left(\frac{n}{N}\right)^{\gamma}T \in \bar{J}:\;\gamma \geq 1,\;0 \leq n \leq N,\;N>1\}$ be a partition of interval $\bar{J} = [0,T]$ with $\Delta t_j = t_j -t_{j-1}$, $j=1,2,\ldots,N$ and $\Delta t=\max \Delta t_j$, $1 \leq j \leq N$. Let $\mathcal{T}_h$ be a regular family of decomposition of $\Omega$ (see, \cite{Ciarlet1978}) into closed $d$-simplexes $\mathbb{T}$ of size $h = \max\{\text{diam}(\mathbb{T}); \mathbb{T} \in \mathcal{T}_h\}$. Further, let $S_h$ be a finite-dimensional subspace of $H_0^1(\Omega)$ with the following approximation property:
\begin{align}\label{approxppt}
    \inf_{\phi_h \in S_h} \{\|\phi-\phi_h\| + h \|\phi-\phi_h\|_1  \} \leq C h^{2}\|\phi\|_{2} \quad \forall\phi \in H_0^1(\Omega)\cap H^{2}(\Omega).
\end{align}
Now, the non-uniform IMEX-L1-FEM for the problem (\ref{variation1}) is to seek $u_h^n \in S_h,\; 0 \leq n \leq N$, such that
\begin{align}
\label{fullydiscrete}
&\begin{cases}
& u^{0}_{h} = P_h u_{0} ,\\
&(D^{\alpha}_{t_n}u_h^n, v_h) + a(t_n;u_h^n,v_h) = \lambda(\mathcal{I}(Eu_h^n),v_h) +(f^n,v_h) \quad \forall v_h \in S_h,\; 1\leq n \leq N,
\end{cases}
\end{align}
where 
\begin{align}
\nonumber &f^n(\cdot):= f(\cdot,t_n),\; a(t_n;\phi,\psi) = \big(\boldsymbol{A}(\cdot,t_n)\nabla\phi,\nabla\psi\big) +  \big(\boldsymbol{b}(\cdot,t_n)\cdot\nabla\phi + c(\cdot,t_n)\phi,\psi\big) \quad \forall \phi,\psi\in H_{0}^{1}(\Omega), \; n=0,1,\ldots,N,  \\
   \label{L1scheme}  &D_{t_n}^{\alpha} \phi(t_{n}) := {\sum_{j=1}^{n}}K^{n,j}_{1-\alpha} \left(\phi(t_j)-\phi(t_{j-1})\right) = {\sum_{j=1}^{n}} \int_{t_{j-1}}^{t_j}k_{1-\alpha}(t_n - s)ds\frac{\left(\phi(t_j)-\phi(t_{j-1})\right)}{\Delta t_j}\\
   \nonumber &\hspace{1.5cm} \approx \; {\sum_{j=1}^{n}} \int_{t_{j-1}}^{t_j}k_{1-\alpha}(t_n - s)\partial_{s}\phi(s)ds \; = \; \int_{0}^{t_n}k_{1-\alpha}(t_n - s)\partial_{s}\phi(s)ds  \; = \; :\partial_{t}^{\alpha} \phi(t_n)
\end{align}
is the well-known L1-formula to approximate the Caputo fractional derivative $\partial_{t}^{\alpha} \phi(t_n).$ Here,
\begin{align}
   \label{discretekernel}& K^{n,j}_{\alpha} = \frac{1}{\Delta t_j}\int_{t_{j-1}}^{t_j}k_{\alpha}(t_n - s)ds = \frac{k_{1+\alpha}(t_n - t_{j-1})-k_{1+\alpha}(t_n - t_{j})}{\Delta t_j}, 
   \end{align}
   and it satisfies
   \begin{align}\label{Kproperty}
    0 \leq K_{1-\alpha}^{n,j-1}< K_{1-\alpha}^{n,j}, \quad 2 \leq j \leq n \leq N .  
   \end{align}
Moreover,
   \begin{align*}
    \nonumber &E\phi^n = \begin{cases}
        \phi^n &:\; \text{fully implicit method,}\\[6pt]
        \begin{cases}
        \phi^0 & \text{if}\; n =1, \\
       (1+\mu_n)\phi^{n-1} - \mu_n \phi^{n-2} & \text{if}\; n \geq 2,
        \end{cases} &:\; \text{IMEX method},\\
    \end{cases}\\
   \nonumber & \mu_{n}=\Delta t_{n}/\Delta t_{n-1}, \quad 2\leq n \leq N. 
\end{align*}
Here, $P_{h}: L^{2}(\Omega)\to S_{h}$ is the $L^{2}$-projection defined by
\begin{align}
    \label{l2projection}& (P_{h}\chi - \chi, v_{h}) = 0\quad \forall v_{h}\in S_{h}, 
\end{align}
and satisfies the following properties (see, \cite{MR2249024}, and \cite{MR1862992})
\begin{align}
    \nonumber & \|\phi-P_{h}\phi\| + h \|\phi-P_{h}\phi\|_1 \leq C h^{2}\|\phi\|_{2} \quad \forall\phi \in H^{2}(\Omega),\\
    \label{l2projectionh1stability}& \|P_{h}\phi\|_{m}\leq \beta_{3} \|\phi\|_{m}\quad \forall \phi\in H^{m}, \; m= 0,1,
\end{align}
where $C$ and $\beta_{3}$ are positive constants independent of the discretizing parameters $h$.

%
\subsection{Stability Analysis}\label{section4}
%
In order to establish the $L^2$- and $H^1$-norms stability of the proposed method, we present a discrete fractional Gr\"{o}nwall inequality related to the discrete fractional differential operator $D_{t_{n}}^{\alpha},\; 1 \leq n\leq N,$  obtained using L1-scheme. In \cite{MR4199354}, Ren et al. have given an improved discrete fractional Gr\"{o}nwall inequality, but it includes a factor $\Gamma{(1-\alpha)}$ which blows up as $\alpha\to 1^{-}.$ Recently, Huang and Stynes \cite{MR4402734} proposed an $\alpha$-robust Gr\"{o}nwall inequality for Alikhanov scheme. Since the Gr\"{o}nwall inequalities proposed in  \cite{MR4402734, MR3790081, MR3904430, MR4199354} are not directly applicable for the proposed IMEX-L1-FEM  when the diffusion coefficient $\boldsymbol{A}$ depends on time $t$ (see Theorem~\ref{H1normstability}), therefore, we present a modified discrete fractional Gr\"{o}nwall inequality in Theorem~\ref{DFGI}. 

Now, we write the properties of discrete kernels $K_{1-\alpha}^{n,j}$ and their complementary discrete kernel
\begin{align}\label{complementarydiscretekernel}
&P_{\alpha}^{n,i} = \frac{1}{K_{1-\alpha}^{i,i}}\begin{cases}  
\displaystyle{\sum_{j=i+1}^n} P_{\alpha}^{n,j} \left(K_{1-\alpha}^{j,i+1} - K_{1-\alpha}^{j,i} \right) & \text{ : }  ~ 1 \leq i \leq n-1,\\
1 & \text{ : }  ~  i = n
\end{cases}
\end{align}
 which are the key tools in deriving our further estimates. 
\begin{lemma}\label{discretekernelproperties}
The discrete kernels $K_{1-\alpha}^{n,j}$ and $P_{\alpha}^{n,j}$ satisfy the following results:
\begin{enumerate}
    \item[(a)] $0 \leq P_{\alpha}^{n,j} \leq \Gamma(2-\alpha) \Delta t_j^{\alpha}, ~ 1 \leq j \leq n. $
    \item[(b)] $\sum_{j=i}^n P_{\alpha}^{n,j} K^{j,i}_{1-\alpha} = 1, ~ 1 \leq i \leq n.$
    \item[(c)] $\sum_{i=1}^n P_{\alpha}^{n,i}k_{1+j\alpha-\alpha} (t_i) \leq k_{1+j\alpha}(t_n) \; \text{ for any non-negative integer } 0 \leq j \leq \lfloor{1/{\alpha}}\rfloor.$
    \item[(d)] $\nu \sum_{j=1}^{n-1}P_{\alpha}^{n,j}E_{\alpha}(\nu t_j^{\alpha}) \leq E_{\alpha}(\nu t_n^{\alpha})-1\;\text{ for any constant } \nu>0,\; \text{provided }$ $ \Delta t_{n-1}\leq \Delta t_{n},\; n \geq 2$,
    \item[(e)] $\sum_{j=1}^{n}P_{\alpha}^{n,j}\;t_j^{\beta-\alpha}\;\leq \; \frac{\Gamma{(1+\beta - \alpha)}}{\Gamma(1+\beta)}\;t_n^\beta\quad \forall \beta\in (0,1),$
\end{enumerate}
where $E_{\alpha}(z) := \sum^{\infty}
_{j=0}\frac{z^j}{\Gamma(j\alpha + 1)}$ is the Mittag-Leffler function. 
\end{lemma}
\begin{proof}
    The proof of $(a),\;(b),\;(c),$ and $(d)$ can be found in Lemma~2.1 and Corollary~4.1 in \cite{MR3790081}. The estimate $(e)$ can be obtained by following the proof of Lemma~5.3 in \cite{MR4246866}.
\end{proof}

\begin{theorem}\label{DFGI}{(Discrete fractional Gr\"{o}nwall inequality).}
    Let $\{v_{n}\}_{n=0}^{N}$, $\{\xi_{n}\}_{n=1}^{N}$, $\{\eta_{n}\}_{n=1}^{N}$ and $\{\zeta_{n}\}_{n=1}^{N}$ be non-negative finite sequences such that 
\begin{align}
    \label{dfgi1} 
     D_{t_n}^{\alpha}(v^{n})^2\;&\leq \; \sum_{i=0}^{n}\lambda_{n-i}^{n}(v^{i})^{2} + v^{n}\xi^{n} + (\eta^{n})^{2}  + (\zeta^{n})^{2},\quad 1\leq n \leq N,
\end{align}
where $\lambda_{j}^{n}\;\geq\; 0,\;0\leq j\leq n,$ and the discrete fractional differential operator $D_{t_{n}}^{\alpha},\; 1 \leq n\leq N,$ is given by (\ref{L1scheme}). If there exists a constant $\Lambda\;>\;0$ such that $\sum_{j=0}^{n}\lambda_{j}^{n}\;\leq\;\Lambda,\;1\leq n \leq N,$ and if $\Delta t_{n-1}\leq\Delta t_{n},\;2\leq n \leq N,$ with the maximum time-step size 
\begin{align} \label{timecondition}
    \Delta t:=\max_{1\leq n \leq N} \Delta t_{n}\;<\; \sqrt[\alpha]{\frac{1}{2\Lambda\Gamma{(2-\alpha)}}},
    \end{align} 
    then 
\begin{align}
    \label{dfgi2} 
    v^{n}\;&\leq \; 2E_{\alpha}(2\Lambda t_{n}^{\alpha})\left[v^{0} + \max_{1\leq j \leq n}\sum_{i=1}^{j}P_{\alpha}^{j,i} \xi^{i} + \sqrt{2t_{n}^{\alpha}}\max_{1\leq j \leq n}\eta^{j}+ \max_{1\leq j \leq n}\sqrt{\sum_{i=1}^{j}P_{\alpha}^{j,i} (\zeta^{i})^2}\right],\quad 1\leq n \leq N.
\end{align}
\end{theorem}
\begin{proof}
    The proof is just a slight variation of the proof of Lemma~2.2 in \cite{MR3790081}, and therefore it is given in the Appendix~\ref{appendix_sec_gronwall_b} for completeness.
\end{proof}
%
\subsection{$L^2$-norm Stability}\label{subsection4.1}
%
As the problem (\ref{fullydiscrete}) is linear in order to ensure stability, it is enough to derive an \textit{a priori} estimate of $u_{h}^{n},\;1\leq n \leq N,$ in terms of $u_{h}^{0}$ and $f^{n}$.
\begin{theorem}\label{stabilitytheorem1}($L^2$-norm stability). Under the condition (\ref{timecondition}) on $\Delta t$, the solution $u_h^n$ of the problem (\ref{fullydiscrete}) satisfies
\begin{align}
\label{L2stability}
&\|u_h^n\| \leq 2C_{0} E_{\alpha}(2 \Lambda_{0} t_n^{\alpha})\left( \|u_h^0\| + \max_{1\leq j \leq n} \sum_{i=1}^jP_{\alpha}^{j,i}\|f^i\|\right), \quad 1 \leq n \leq N,  
\end{align}  
where
\begin{align*}
&  \Lambda_{0} := \begin{cases}
    \frac{\beta_1^2}{\beta_0} + 2 |\lambda| \beta_2 &:\text{ fully-implicit method,} \\[4pt]
    \frac{\beta_1^2}{\beta_0} + 2 |\lambda|(1+ 2 \mu ) \beta_2 &:\text{ implicit-explicit method,}
\end{cases}\\
& \mu:=\max_{2\leq n \leq N}\mu_{n},\; {and}\;  C_{0} := 2.
\end{align*}
\end{theorem}
\begin{proof}
The definition of discrete fractional differential operator $D_{t_n}^\alpha$ along with the property (\ref{Kproperty}) yields
\begin{align}
\nonumber \left(D_{t_{n}}^{\alpha}\phi^{n},\phi^{n}\right)  =& \left(\sum_{j=1}^n K_{1-\alpha}^{n,j}(\phi^j - \phi^{j-1}), \phi^{n} \right)\\
\nonumber =&\; K_{1-\alpha}^{n,n}\|\phi^n\|^2-\sum_{j=1}^{n-1} (K_{1-\alpha}^{n,j+1} - K_{1-\alpha}^{n,j})(\phi^j , \phi^{n} )  - K_{1-\alpha}^{n,1}(\phi^0 , \phi^{n} )\\
\nonumber \geq&\; K_{1-\alpha}^{n,n}\|\phi^n\|^2 -\frac{1}{2}\sum_{j=1}^{n-1} (K_{1-\alpha}^{n,j+1} - K_{1-\alpha}^{n,j})\|\phi^j\|^2  - \frac{1}{2}\sum_{j=1}^{n-1} (K_{1-\alpha}^{n,j+1} - K_{1-\alpha}^{n,j})\|\phi^n\|^2  \\
\nonumber&\; - \frac{1}{2}K_{1-\alpha}^{n,1}\|\phi^0\|^2  -\frac{1}{2}K_{1-\alpha}^{n,1}\|\phi^n\|^2 \\
\label{dfdoestimate} =&\; \frac{1}{2}\sum_{j=1}^n K_{1-\alpha}^{n,j}(\|\phi^j\|^2  - \|\phi^{j-1}\|^2 ) = \frac{1}{2}D_{t_{n}}^{\alpha}\|\phi^{n}\|^{2}.
\end{align} 
Now, choose $v_h = u_h^n$ in (\ref{fullydiscrete}), and then apply  estimate (\ref{dfdoestimate}), (\ref{a1bounded1}) and the Cauchy-Schwarz inequality appropriately to obtain 
\begin{align*}
 \frac{1}{2}D_{t_n}^{\alpha}\|u_h^n\|^2 + \beta_0\|u_h^n\|_1^2 &\leq (D_{t_n}^{\alpha}u^n_h, u_h^n) + a_0(t_n;u_h^n, u_h^n) \\
 &= \lambda(\mathcal{I}(E u^n_h), u_h^n) + (f^n, u_h^n) - a_1(t_n;u_h^n, u_h^n) \\
 & \leq |\lambda| \beta_2\|Eu_h^n\|\;\|u_h^n\| + \frac{\beta_1^2}{2\beta_0}\|u_h^n\|^2 + \|f^n\|\;\|u_h^n\| + \frac{\beta_0}{2}\|u_h^n\|_1^2.
\end{align*} 
Thus,
\begin{align}\label{stability2}
 \nonumber D_{t_n}^{\alpha}\|u_h^n\|^2  &\leq  2|\lambda| \beta_2\|Eu_h^n\|\;\|u_h^n\| + \frac{\beta_1^2}{\beta_0}\|u_h^n\|^2 + 2\|f^n\|\;\|u_h^n\| \\
 &\leq  \sum_{i=0}^{n}\lambda_{n-i}^{n}\|u_{h}^{i}\|^2+ 2\|f^n\|\;\|u_h^n\|,
\end{align}
where, in the case of 
\begin{enumerate}
    \item[(i)] fully implicit method, i.e., $E\phi^n = \phi^n$, 
\begin{align}
\nonumber& \lambda_{n-i}^n := \begin{cases}
\frac{\beta_1^2}{\beta_0} + 2|\lambda| \beta_2 &: \; 1\leq n \leq N\;\&\; i=n, \\
0 &:\; 1\leq n \leq N\;\&\; 0\leq i \leq n-1,
\end{cases} \\
\nonumber & \displaystyle{\sum_{i=0}^n \lambda_{n-i}^n} = \frac{\beta_1^2}{\beta_0} + 2 |\lambda| \beta_2 = \max_{1 \leq n \leq N} \sum_{i=0}^n \lambda_{n-i}^n =:\Lambda_{0}.
\end{align}
\item[(ii)] IMEX method, i.e., 

$E\phi^n =\begin{cases}
        \phi^0 &: \; n =1,\\
        (1+\mu_n)\phi^{n-1} - \mu_n \phi^{n-2} &: \; 2\leq n \leq N,
\end{cases}$

\begin{align}
\nonumber& \lambda_{n-i}^n := \begin{cases}
\frac{\beta_1^2}{\beta_0} + |\lambda| \beta_2 &: \;  n\;=\;1\;\&\; i=n, \\
|\lambda| \beta_2 &: \; n\;=\;1\;\&\; i=n-1, \\ 
\frac{\beta_1^2}{\beta_0} + |\lambda| (1+2 \mu_n)\beta_2 &: \;  2\leq n \leq N\;\&\; i=n, \\
|\lambda|(1+\mu_n)\beta_2 &: \;  2\leq n \leq N\;\&\; i=n-1, \\
|\lambda| \mu_n \beta_2 &: \; 2\leq n \leq N\;\&\; i=n-2, \\
0 &: \;  3\leq n \leq N\;\&\; 0 \leq i \leq n-3, \\
    \end{cases}\\
\nonumber  \displaystyle{\sum_{i=0}^n \lambda_{n-i}^n} &= \frac{\beta_1^2}{\beta_0} + 2 |\lambda|(1+ 2 \mu_n) \beta_2 \;\leq\; \frac{\beta_1^2}{\beta_0} + 2 |\lambda|(1+ 2\mu) \beta_2 \\
\nonumber&=\; \max_{1 \leq n \leq N} \sum_{i=0}^n \lambda_{n-i}^n =:\Lambda_{0}.
\end{align}
\end{enumerate}
In the case of IMEX method, to obtain the estimate (\ref{stability2}), we have used $\|E\phi^{1}\|\|\phi^1\|\leq\frac{1}{2}\|\phi^{0}\|^2 + \frac{1}{2}\|\phi^{1}\|^2$ and $\|E\phi^{n}\|\|\phi^n\|\leq \frac{\mu_n}{2}\|\phi^{n-2}\|^2 + \frac{1+\mu_n}{2}\|\phi^{n-1}\|^2 + \frac{1+2\mu_n}{2}\|\phi^{n}\|^2,\;2\leq n \leq N.$

Now, an appeal to the generalized discrete fractional Gr\"{o}nwall inequality (Theorem~\ref{DFGI}) yields the estimate (\ref{L2stability}), and this completes the rest of the proof.
\end{proof}
%
\subsection{$H^1$-norm Stability}\label{subsection5}
%
This subsection deals with the stability of the non-uniform IMEX-L1-FEM (\ref{fullydiscrete}) with respect to $H^{1}$-norm. Now, define the discrete linear operator $\mathcal{L}_{0h}^{n}: S_{h}\to S_{h},\;1\leq n \leq N,$ (see, \cite{MR2249024}, and \cite{MR3957890}) as
\begin{align}
    \label{deo}& (\mathcal{L}_{0h}^{n}\psi, v_{h})=a_{0}(t_{n};\psi,v_{h}) \quad \forall \psi,v_{h}\in S_{h}, \; 1\leq n \leq N.
\end{align}
This satisfies the following estimate
\begin{align}
   \label{deoestimate} & a_{0}(t_{n};\phi,\mathcal{L}_{0h}^{n}\phi)=\|\mathcal{L}_{0h}^{n}\phi\|^{2}\quad \forall \phi\in S_{h}, \; 1\leq n \leq N.
\end{align}
Moreover, for each $n,\; 1\leq n \leq N,$ inequalities (\ref{bounded}) and (\ref{coercive}) yield that $\vertiii{\cdot}_n:=\sqrt{a_{0}(t_{n};\cdot,\cdot)}$ is a norm on $H_{0}^{1}(\Omega)$ which is equivalent to $\|\cdot\|_1$-norm, i.e.,
\begin{align}
    \label{equivalentnorms}&\beta_{0}\|\phi\|_{1}^{2}\leq \vertiii{\phi}_n^{2}\leq \gamma_{0}\|\phi\|_{1}^{2}\quad \forall \phi\in H_{0}^{1}.
\end{align}
Furthermore, forming an inner-product between the discrete fractional differential operator $D_{t_n}^{\alpha} \phi^{n}$ and $\mathcal{L}_{0h}^{n}\phi^{n},$ for  $1\leq n\leq N,$ a use of \eqref{deo} shows 
\begin{align}
\nonumber &\left(D_{t_{n}}^{\alpha}\phi^{n},\mathcal{L}_{0h}^{n}\phi^{n}\right)  = (\sum_{j=1}^n K_{1-\alpha}^{n,j}(\phi^j - \phi^{j-1}), \mathcal{L}_{0h}^{n}\phi^{n} )\\
\nonumber&= K_{1-\alpha}^{n,n}(\phi^n, \mathcal{L}_{0h}^{n}\phi^{n} )-\sum_{j=1}^{n-1} (K_{1-\alpha}^{n,j+1} - K_{1-\alpha}^{n,j})(\phi^j , \mathcal{L}_{0h}^{n}\phi^{n} ) - K_{1-\alpha}^{n,1}(\phi^0 , \mathcal{L}_{0h}^{n}\phi^{n} )\\
\nonumber&= K_{1-\alpha}^{n,n}a_0(t_n;\phi^n , \phi^{n} ) -\sum_{j=1}^{n-1} (K_{1-\alpha}^{n,j+1} - K_{1-\alpha}^{n,j})a_0(t_n;\phi^j , \phi^{n} )  - K_{1-\alpha}^{n,1}a_0(t_n;\phi^0 , \phi^{n} ).
 \end{align}
An application of \eqref{dfdoestimate} with the property  of  (\ref{Kproperty})  yields
 \begin{align}
 \nonumber &\left(D_{t_{n}}^{\alpha}\phi^{n},\mathcal{L}_{0h}^{n}\phi^{n}\right) \geq 
K_{1-\alpha}^{n,n}\vertiii{\phi^n}^2_n  - \frac{1}{2}(K_{1-\alpha}^{n,n} - K_{1-\alpha}^{n,1})\vertiii{\phi^n}^2_n -\frac{1}{2}\sum_{j=1}^{n-1} (K_{1-\alpha}^{n,j+1} - K_{1-\alpha}^{n,j})\vertiii{\phi^j}^2_n  \\
\nonumber&- \frac{1}{2}K_{1-\alpha}^{n,1}\vertiii{\phi^0}^2_n -\frac{1}{2}K_{1-\alpha}^{n,1}\vertiii{\phi^n}^2_n\\
\nonumber& = \frac{1}{2}\sum_{j=1}^n K_{1-\alpha}^{n,j}(\vertiii{\phi^j}^2_n - \vertiii{\phi^{j-1}}^2_n)\\
\label{pollutionterm}&=\frac{1}{2}D_{t_{n}}^{\alpha}\vertiii{\phi^{n}}_n^{2} -\frac{1}{2}\sum_{j=1}^{n-1} (K_{1-\alpha}^{n,j+1} - K_{1-\alpha}^{n,j})\left(\vertiii{ \phi^j}^2_n -\vertiii{ \phi^j}^2_{j}\right)- \frac{1}{2}K_{1-\alpha}^{n,1}\left(\vertiii{ \phi^0}^2_n -\vertiii{ \phi^0}^2_{0}\right).
\end{align}
Here, we have used the property (\ref{Kproperty}) and $a_{0}(t_n;\phi,\psi) \leq \frac{1}{2}a_{0}(t_n;\phi,\phi) + \frac{1}{2}a_{0}(t_n;\psi,\psi) = \frac{1}{2}\vertiii{\phi}_n^{2} + \frac{1}{2}\vertiii{\psi}_n^{2},\;\forall \phi,\psi\in H_0^1(\Omega),\;1\leq n \leq N$. 
\begin{theorem}\label{H1normstability}($H^1$-norm stability). Under the condition (\ref{timecondition}) on $\Delta t$, the solution $u_h^n$ of the problem (\ref{fullydiscrete}) satisfies
\begin{align*}
& \|u_h^n\|_{1}\leq
2C_{1}E_{\alpha}\left(2\Lambda_{1}t_n^{\alpha}\right)\left(\|u_h^0\|_{1} + \max_{1 \leq j \leq n }\|f^j\|\right),\; 1\leq n \leq N, 
\end{align*}
where 
\begin{align*}
&  \Lambda_{1} := \begin{cases}
    2\frac{\beta_1^2+\lambda^2 \beta_2^2}{\beta_0} + \frac{T^{1-\alpha}}{\Gamma(2-\alpha)} &:\text{ fully-implicit method,} \\[4pt]
   2\frac{\beta_1^2+2\lambda^2 \beta_2^2((1+\mu)^2+\mu^2)}{\beta_0} + \frac{T^{1-\alpha}}{\Gamma(2-\alpha)} &:\text{ implicit-explicit method,}
\end{cases}\\
& \mu:=\max_{2\leq n \leq N}\mu_n,\; {and}\; C_{1}:=\max\left\{\sqrt{\frac{\gamma_{0}}{\beta_{0}}},\; \sqrt{\frac{2T^\alpha}{\beta_{0}}}\right\}.
\end{align*}
\end{theorem}
\begin{proof}
Take  $v_h = \mathcal{L}_{0h}^n u_h^n$ in (\ref{fullydiscrete}) and then apply estimates (\ref{deoestimate})-(\ref{pollutionterm}), (\ref{Lipschitz}),  (\ref{a1bounded1}), (\ref{integralbounded}), the equivalent norms (\ref{equivalentnorms}) and the Cauchy-Schwarz inequality appropriately to obtain 
\begin{align*}
&\frac{1}{2}D_{t_n}^{\alpha}\vertiii{u_{h}^{n}}_n^{2} +\|\mathcal{L}_{0h}^n u_h^n\|^2 \leq \; (D_{t_n}^{\alpha}u_h^n,\mathcal{L}_{0h}^n u_h^n) + a_{0}(t_{n};u_h^n,\mathcal{L}_{0h}^n u_h^n) + \frac{1}{2}\sum_{i=1}^{n-1} (K_{1-\alpha}^{n,i+1} - K_{1-\alpha}^{n,i})\left(\vertiii{ u_h^i}^2_n -\vertiii{ u_h^i}^2_{i}\right)\\
&\hspace{1cm} \;+ \frac{1}{2}K_{1-\alpha}^{n,1}\left(\vertiii{ u_h^0}^2_n -\vertiii{ u_h^0}^2_{0}\right) \\
&\hspace{1cm} =  \lambda (\mathcal{I}(Eu^n_h),\mathcal{L}_{0h}^n u_h^n) + (f^n,\mathcal{L}_{0h}^n u_h^n)-a_1(t_n;u_h^n,\mathcal{L}_{0h}^n u_h^n) + \frac{1}{2}\sum_{i=1}^{n-1} (K_{1-\alpha}^{n,i+1} - K_{1-\alpha}^{n,i})\left(\vertiii{ u_h^i}^2_n -\vertiii{ u_h^i}^2_{i}\right)\\
&\hspace{1cm}\;+ \frac{1}{2}K_{1-\alpha}^{n,1}\left(\vertiii{ u_h^0}^2_n -\vertiii{ u_h^0}^2_{0}\right) \\
&\hspace{1cm} \leq |\lambda| \beta_2\|E u_h^n\|~\| \mathcal{L}_{0h}^n u_h^n\| + \|f^n\|\|\mathcal{L}_{0h}^n u_h^n\| + \beta_1\|u_h^n\|_1\|\mathcal{L}_{0h}^n u_h^n\|\\
&\hspace{1cm}\;+ \frac{1}{2}\sum_{i=1}^{n-1} (K_{1-\alpha}^{n,i+1} - K_{1-\alpha}^{n,i})\left(a_0(t_n;u_h^i,u_h^i)-a_0(t_i;u_h^i,u_h^i)\right)+ \frac{1}{2}K_{1-\alpha}^{n,1}\left(a_0(t_n;u_h^0,u_h^0)-a_0(t_0;u_h^0,u_h^0)\right) \\
&\hspace{1cm} \leq   \lambda^2 \beta_2^2\|E u_h^n\|^2  + \beta_1^2\|u_h^n\|_1^2 + \frac{1}{2}\|f^n\|^2 + \|\mathcal{L}_{0h}^n u_h^n\|^2+ \frac{L}{2\beta_0}\sum_{i=1}^{n-1} (K_{1-\alpha}^{n,i+1} - K_{1-\alpha}^{n,i})(t_n-t_i)\vertiii{u_h^i}_n^{2}\\
&\hspace{1cm}\;+ \frac{L}{2\beta_0}K_{1-\alpha}^{n,1}(t_n-t_0)\vertiii{u_h^0}_n^{2}.
\end{align*}
Since, $\|E u_h^n\|^2 \leq \|E u_h^n\|^2_1 \leq \frac{1}{\beta_0}\vertiii{E u_h^n}^2_n$ and $\| u_h^n\|^2_1 \leq \frac{1}{\beta_0}\vertiii{ u_h^n}^2_n$, we establish
\begin{align*}
D_{t_n}^{\alpha}\vertiii{u_h^n}_n^2 \leq &\; 2\frac{\lambda^2 \beta_2^2}{\beta_0}\vertiii{E u_h^n}_n^2 
 + 2\frac{\beta_1^2}{\beta_0}\vertiii{u_h^n}_n^2  + \|f^n\|^2+ \frac{L}{\beta_0}\sum_{i=1}^{n-1} (K_{1-\alpha}^{n,i+1} - K_{1-\alpha}^{n,i})(t_n-t_i)\vertiii{u_h^i}_n^{2}\\
 &+ \frac{L}{\beta_0}K_{1-\alpha}^{n,1}(t_n-t_0)\vertiii{u_h^0}_n^{2}\\
 \leq & \sum_{i=0}^{n}\lambda_{n-i}^{n}\vertiii{u_{h}^{i}}_n^2+ \|f^n\|^2,
\end{align*}
where, in the case of 
\begin{enumerate}
    \item[(i)] fully implicit method, i.e., $E\phi^n = \phi^n$, 
\begin{align}
\nonumber& \lambda_{n-i}^n := \begin{cases}
2\frac{\beta_1^2+\lambda^2 \beta_2^2}{\beta_0}  &: \; 1\leq n \leq N\;\&\; i=n, \\
\frac{L}{\beta_0}(K_{1-\alpha}^{n,i+1} - K_{1-\alpha}^{n,i})(t_n-t_i) &:\; 1\leq n \leq N\;\&\; 1\leq i \leq \;n-1, \\
\frac{L}{\beta_0}K_{1-\alpha}^{n,1} (t_n-t_0) &:\; 1\leq n \leq N\;\&\; i = 0, \\
\end{cases} \\
\nonumber \displaystyle{\sum_{i=0}^n \lambda_{n-i}^n} &= 2\frac{\beta_1^2+\lambda^2 \beta_2^2}{\beta_0} + \frac{L}{\beta_0}\sum_{i=1}^{n-1} (K_{1-\alpha}^{n,i+1} - K_{1-\alpha}^{n,i})(t_n-t_i) + \frac{L}{\beta_0}K_{1-\alpha}^{n,1} (t_n-t_0)\\
\nonumber&\leq 2\frac{\beta_1^2+\lambda^2 \beta_2^2}{\beta_0}  + \frac{LT^{1-\alpha}}{\beta_0\Gamma(2-\alpha)}= \max_{1 \leq n \leq N} \sum_{i=0}^n \lambda_{n-i}^n =:\Lambda_{1}.
\end{align}
\item[(ii)] IMEX method, i.e., 

$E\phi^n =\begin{cases}
        \phi^0 &: \; n =1,\\
        (1+\mu_n)\phi^{n-1} - \mu_n \phi^{n-2} &: \; 2\leq n \leq N,
\end{cases}$
\begin{align}
\nonumber& \lambda_{n-i}^n := \begin{cases}
2\frac{\beta_1^2}{\beta_0} &: \;  n\;=\;1\;\&\; i=1, \\[4pt]
2 \frac{\lambda^2\beta_2^2}{\beta_0} + \frac{L}{\beta_0}K_{1-\alpha}^{1,1} (t_1-t_0)&: \; n\;=\;1\;\&\; i=0, \\[4pt]
2\frac{\beta_1^2}{\beta_0} &: \;  2\leq n \leq N\;\&\; i=n, \\[4pt]
4 (1+\mu_n)^2 \frac{\lambda^2\beta_2^2}{\beta_0} + \frac{L}{\beta_0}(K_{1-\alpha}^{n,n} - K_{1-\alpha}^{n,n-1})(t_n-t_{n-1})&: \;  2\leq n \leq N\;\&\; i=n-1, \\[4pt]
4  \frac{\lambda^2\beta_2^2\mu^2_n}{\beta_0}+ \frac{L}{\beta_0}(K_{1-\alpha}^{n,n-1} - K_{1-\alpha}^{n,n-2})(t_n-t_{n-2}) &: \; 2\leq n \leq N\;\&\; i=n-2, \\[4pt]
\frac{L}{\beta_0}(K_{1-\alpha}^{n,i+1} - K_{1-\alpha}^{n,i})(t_n-t_i) &: \;  3\leq n \leq N\;\&\; 1 \leq i \leq n-3, \\[4pt]
\frac{L}{\beta_0}K_{1-\alpha}^{n,1} (t_n-t_0) &: \;  3\leq n \leq N\;\&\; i=0, \\
    \end{cases}\\
\nonumber \displaystyle{\sum_{i=0}^n \lambda_{n-i}^n} &=     2\frac{\beta_1^2+2\lambda^2 \beta_2^2((1+\mu_n)^2+\mu_n^2)}{\beta_0}+ \frac{L}{\beta_0}\sum_{i=1}^{n-1} (K_{1-\alpha}^{n,i+1} - K_{1-\alpha}^{n,i})(t_n-t_i) + \frac{L}{\beta_0}K_{1-\alpha}^{n,1} (t_n-t_0)\\
\nonumber&\leq 2\frac{\beta_1^2+2\lambda^2 \beta_2^2((1+\mu)^2+\mu^2)}{\beta_0}   + \frac{LT^{1-\alpha}}{\beta_0\Gamma(2-\alpha)}    \;=\; \max_{1 \leq n \leq N} \sum_{i=0}^n \lambda_{n-i}^n =:\Lambda_{1}.
\end{align}
\end{enumerate}
Here, we have used
\begin{align} \label{pollutionbound}
\nonumber \sum_{i=1}^{n-1} (K_{1-\alpha}^{n,i+1} - K_{1-\alpha}^{n,i})(t_n-t_i) &+ K_{1-\alpha}^{n,1} (t_n-t_0)\\
\nonumber &=  \sum_{i=1}^{n-1} \left(K_{1-\alpha}^{n,i+1} (t_n-t_i)-  K_{1-\alpha}^{n,i}(t_n-t_{i-1})\right) +K_{1-\alpha}^{n,1} (t_n-t_0)+ \sum_{i=1}^{n-1}K_{1-\alpha}^{n,i}(t_i - t_{i-1})\\
\nonumber&=  \sum_{i=1}^{n}K_{1-\alpha}^{n,i}(t_i - t_{i-1})=\sum_{i=1}^{n}\int_{t_{i-1}}^{t_i}k_{1-\alpha}(t_n-s)ds = \int_0^{t_n}k_{1-\alpha}(t_n-s)ds\\
& =  \frac{t_n^{1-\alpha}}{\Gamma(2-\alpha)} \leq \frac{T^{1-\alpha}}{\Gamma(2-\alpha)}.
\end{align}
Apply the generalized discrete fractional Grönwall inequality (Theorem~\ref{DFGI}), and then use estimate (\ref{equivalentnorms}) to complete the rest of the proof.
\end{proof}
\section{Error analysis}\label{section5}
In this section, we derive first some auxiliary results and then establish optimal error estimates. Now, at any temporal grid point $t_n$, variational problem (\ref{variation1}) implies for $1\leq n \leq N$
\begin{align}\label{variational2}
(D_{t_n}^{\alpha}u(t_n),v) + a(t_n; u(t_n),v) = \lambda (\mathcal{I}u(t_n),v) +(f(t_n),v) + (\Upsilon^n+ r^n, v) \quad \forall\;v \in H_0^1(\Omega),    
\end{align}
where $\Upsilon^n:= (D^{\alpha}_{t_n}u(t_n) - \partial^{\alpha}_{t}u(t_n)), $ $r^n:= \lambda \mathcal{I}(u(t_n) - Eu(t_n))$. After subtracting (\ref{fullydiscrete}) from (\ref{variational2}), we obtain the following error equation:
\begin{align}\label{errorequation}
(D_{t_n}^{\alpha}e^n_h,v_h) + a(t_n; e^n_h,v_h) = \lambda (\mathcal{I}(Ee^n_h),v_h) + (\Upsilon^n+ r^n, v_h) \quad \forall\;v_h \in S_h, \; 1\leq n \leq N,   
\end{align}
 where $e^n_h:= u(t_n)-u_h^n$ denotes the error between the exact solution $u(t_n)$ and the approximate solution $u_h^n$ at time level $n.$ To obtain an optimal error estimate, we decompose the error $e^n_h$ further as follows:
$$e^n_h = \eta^n + \theta^n, \; \eta^n = u(t_n)- \mathcal{R}_h(t_n)u(t_n), \; \theta^n =  \mathcal{R}_h(t_n)u(t_n) - u_h^n,$$
where $\mathcal{R}_{h}(t): H_{0}^{1}(\Omega)\to S_{h}, \; \forall t\in [0, T] $ is an elliptic projection defined by, (see, \cite{MR2249024}),
\begin{align}
   \label{ellipticprojection}&a_{0}(t;\mathcal{R}_{h}(t)u(t) - u(t),\psi_{h}) = 0\quad \forall \psi_{h}\in S_{h}.
   \end{align}
This elliptic projection satisfies
   \begin{align}
   \label{ellipticprojH1bound}& \| \mathcal{R}_h(t)\phi\|_{1}\leq C\|\phi\|_{1}\quad\forall \phi\in H_{0}^{1}.
\end{align}
Set $\eta(t)= u(t)- \mathcal{R}_h(t)u(t)$ then the following estimates hold, (see, \cite{MR2249024}),
\begin{align}\label{optimaletaestimate}
   \|\eta\|+h\|\eta\|_{1}\leq C h^{s}\|u(t)\|_{s},\;s=1,2,\\
\label{optimaletaestimate2}   \|\eta_t\| + h\|\eta_t\|_1 \leq Ch^s\|u_t(t)\|_s,\;s=1,2.
\end{align} 
In order to establish the final error estimate, it is enough to estimate $\theta^n$. The application of elliptic projection (\ref{ellipticprojection}) shows $\forall\;v_h \in S_h, \; 1\leq n \leq N,$ 
\begin{align}\label{erroreqtheta}
 (D^{\alpha}_{t_n}\theta^n, v_h) + a_0(t_n;\theta^n, v_h)= \lambda (\mathcal{I}(E\theta^n), v_h)+(\lambda \mathcal{I}(E\eta^n) + \Upsilon^n + r^n - D^{\alpha}_{t_n}\eta^n, v_h ) - a_1(t_n;\theta^n + \eta^n, v_h).  
\end{align}
\begin{lemma}\label{thetaestimate_l2norm}
Under the condition (\ref{timecondition}), there holds
\begin{align*}
 &\|\theta^n\| \leq 2D_{0}E_{\alpha}(2\Lambda_{0} t_n^{\alpha})\left(\|\theta^0\| + \max_{1\leq j \leq n}\left(\sum_{i=1}^jP_{\alpha}^{j,i}(\|\Upsilon^i+r^i\| + \|D_{t_i}^{\alpha}\eta^i\| )\right)+ \max_{1\leq j \leq n}\|\eta^j\| + \max_{1\leq j \leq n}\|E\eta^j\| \right),\; 1\leq n \leq N,
 \end{align*}
where $D_{0}:=\max\left(2,4 |\lambda| \beta_2 T^\alpha, 2\beta_1\sqrt{\frac{T^\alpha}{\beta_0}}\right),$ and $\Lambda_{0}$ is same as in Theorem~\ref{stabilitytheorem1}.
\end{lemma}
\begin{proof}
Choose $v_h = \theta^n$ in (\ref{erroreqtheta}) then apply the estimate (\ref{dfdoestimate}) and the coercivity (\ref{coercive}) of the bi-linear form $a_0(t;,\cdot,\cdot)$, to arrive at
\begin{align*}    D_{t_n}^{\alpha}\|\theta^n\|^2  + 2\beta_0 \|\theta^n\|_1^2  \leq &\; 2\lambda (\mathcal{I}(E\theta^n), \theta^n) + 2\Big(\lambda \mathcal{I}(E\eta^n) + \Upsilon^n + r^n - D^{\alpha}_{t_n}\eta^n, \theta^n \Big)\\
     &- 2a_1(t_n;\theta^n + \eta^n, \theta^n).
\end{align*}
An appropriate application the Cauchy-Schwarz inequality with AM-GM inequality  $(a+b)^2 \leq 2(a^2 +b^2), \; a,b\geq 0$ and (\ref{a1bounded2}) yields for $1\leq n \leq N$
\begin{align*}
D_{t_n}^{\alpha}\|\theta^n\|^2  + 2\beta_0\|\theta^n\|_1^2\leq &\; 2|\lambda|\beta_2\|E\theta^n\|~\|\theta^n\| + 2\|\lambda \mathcal{I}(E\eta^n) + \Upsilon^n + r^n - D^{\alpha}_{t_n}\eta^n\|\| \theta^n \| \\
&+ \frac{2\beta_1^2}{\beta_0} \left(\|\theta^n\|^2 + \|\eta^n\|^2\right)+ 2\beta_0\|\theta^n\|_1^2 .
\end{align*}
Hence as $\|\theta^n\|^2_1 \geq 0$, we obtain
\begin{align*}
D_{t_n}^{\alpha}\|\theta^n\|^2 \leq 2|\lambda| \beta_2\|E\theta^n\|~\|\theta^n\| +  \frac{2\beta_1^2}{\beta_0} \|\theta^n \|^2+ 2\|\lambda \mathcal{I}(E\eta^n) + \Upsilon^n + r^n - D^{\alpha}_{t_n}\eta^n\|\| \theta^n \| + \frac{2\beta_1^2}{\beta_0} \|\eta^n\|^2   
\end{align*}
for $1\leq n \leq N$. Apply the generalized discrete fractional Gr\"{o}nwall inequality (Theorem \ref{DFGI}) in the previous inequality to arrive for $1\leq n \leq N$ at,
\begin{align*}
& \|\theta^n\| \leq  2E_{\alpha}(2\Lambda_{0} t_n^{\alpha})\Bigg(\|\theta^0\|+2\max_{1\leq j \leq n}\sum_{i=1}^jP_{\alpha}^{j,i}(\|\lambda \mathcal{I}(E\eta^i)\| + \|\Upsilon^i\| + \|r^i\| + \|D^{\alpha}_{t_i}\eta^i\| )+ 2\beta_1\sqrt{\frac{t_n^{\alpha}}{\beta_0}}\max_{1\leq j \leq n}\|\eta^j\| \Bigg)   \\
 & \; \leq  2E_{\alpha}(2\Lambda_{0} t_n^{\alpha})\Bigg(\|\theta^0\| + 2\max_{1\leq j \leq n}\sum_{i=1}^jP_{\alpha}^{j,i}(\|\Upsilon^i \|+ \|r^i\| + \| D^{\alpha}_{t_i}\eta^i\| ) + 4|\lambda|\beta_2 t_n^\alpha \max_{1\leq j \leq n}\|E\eta^j \|+ 2\beta_1\sqrt{\frac{t_n^{\alpha}}{\beta_0}}\max_{1\leq j \leq n}\|\eta^j\| \Bigg)\\
 & \; \leq 2D_0 E_{\alpha}(2\Lambda_{0} t_n^{\alpha}) \left(\|\theta^0\| +\max_{1\leq j \leq n}\sum_{i=1}^jP_{\alpha}^{j,i}(\|\Upsilon^i \|+ \|r^i\| + \| D^{\alpha}_{t_i}\eta^i\| ) 
 + \max_{1\leq j \leq n} \|\eta^j\|  + \max_{1\leq j \leq n}  \|E\eta^j\| \right),
\end{align*}
where $D_{0}:=\max\left(2,4|\lambda|\beta_2 T^\alpha, 2\beta_1\sqrt{\frac{T^\alpha}{\beta_0}}\right),$ and $\Lambda_{0}$ is same as in Theorem~\ref{stabilitytheorem1}. This completes the rest of the proof.
\end{proof}
\begin{lemma}\label{thetaestimate_h1norm}
Under the assumptions in Theorem \ref{Regularity condition} and maximum temporal grid size restriction (\ref{timecondition}), there holds
\begin{align}
 \nonumber&\|\theta^n\|_1\; \leq \; 2D_1 E_{\alpha}(2\Lambda_1 t_n^{\alpha})\Big(\|\theta^0\|_1 + \max_{1\leq j \leq n}\sum_{i=1}^jP_{\alpha}^{j,i}\left(\|\Upsilon^i\|_1 + \|r^i\|_1 \right) + \max_{1\leq j \leq n}\sum_{i=1}^jP_{\alpha}^{j,i}\|D^{\alpha}_{t_i}\eta^i\|_1+\max_{1\leq j \leq n}\|b\|_{\infty} \|\eta^j\|_1\\
 \nonumber&\hspace{5cm}+ \max_{1\leq j \leq n} \| \eta^j\| +\max_{1\leq j \leq n}\|Ee_h^j\| \Big),\;1\leq n \leq N,
\end{align}
where $\Lambda_{1}=\frac{(\|b\|_{\infty}+\|c\|_{\infty})^2 }{\beta_0} + \frac{LT^{1-\alpha}}{\beta_0\Gamma(2-\alpha)} $ and $D_1=\max\left\{\sqrt{\frac{\gamma_0}{\beta_0}},\; \frac{2\gamma_0\beta_3}{\beta_0}, \;\frac{1}{\sqrt{\beta_0}}\sqrt{2T^{\alpha}}, \;\frac{\|c\|_{\infty}}{\sqrt{\beta_0}}\sqrt{2T^{\alpha}},\;\frac{|\lambda|\beta_2}{\sqrt{\beta_0}}\sqrt{2T^{\alpha}} \;\right\}$.
\end{lemma}
\begin{proof}
Set $v_h = \mathcal{L}_{0h}^n \theta^n$ in (\ref{erroreqtheta}) and use the $L^2$-projection (\ref{l2projection}) yields
\begin{align*}
 (D_{t_n}^{\alpha}\theta^n, \mathcal{L}_{0h}^n \theta^n)+ a_0(t_n;\theta^n,\mathcal{L}_{0h}^n \theta^n)=  \lambda(\mathcal{I}(Ee_h^n),\mathcal{L}_{0h}^n \theta^n)+(P_h (\Upsilon^n + r^n - D^{\alpha}_{t_n}\eta^n),\mathcal{L}_{0h}^n \theta^n) -a_1(t_n;e_h^n,\mathcal{L}_{0h}^n \theta^n). 
\end{align*}
Apply estimates (\ref{pollutionterm}) and relations (\ref{Lipschitz}), (\ref{deoestimate}), (\ref{deo}) and equivalent norms (\ref{equivalentnorms}) to get the following estimate
\begin{align}\label{H1split}
\nonumber&\frac{1}{2}D_{t_n}^{\alpha}\vertiii{\theta^n}_n^2 + \|\mathcal{L}_{0h}^n \theta^n\|^2 \;\leq\: |\lambda|(\mathcal{I}(Ee_h^n),\mathcal{L}_{0h}^n \theta^n)+a_0(t_n;\theta^n, P_h (\Upsilon^n + r^n - D^{\alpha}_{t_n}\eta^n)) \\
&\hspace{2cm}-a_1(t_n;e_h^n,\mathcal{L}_{0h}^n \theta^n)+ \frac{L}{2\beta_0}\sum_{i=1}^{n-1} (K_{1-\alpha}^{n,i+1} - K_{1-\alpha}^{n,i})|t_n-t_i|\vertiii{\theta^i}_n^{2}+ \frac{L}{2\beta_0}K_{1-\alpha}^{n,1}|t_n-t_0|\vertiii{\theta^0}_n^{2}.
\end{align}
Now, by using (\ref{bounded}), (\ref{a1bounded1}), (\ref{integralbounded}), the Cauchy-Schwarz inequality and the AM-GM inequality appropriately, we obtain
\begin{align*}
\frac{1}{2}D_{t_n}^{\alpha}\vertiii{\theta^n}_n^2 &+ \|\mathcal{L}_{0h}^n \theta^n\|^2\leq (|\lambda|\beta_2\|Ee_h^n\| +\|b\|_{\infty} \|\nabla e_h^n\|+ \|c\|_{\infty} \| e_h^n\|)\|\mathcal{L}_{0h}^n \theta^n\| + \gamma_0\|P_h (\Upsilon^n + r^n -D^{\alpha}_{t_n}\eta^n)\|_1 \|\theta^n\|_1 \\
&+ \frac{L}{2\beta_0}\sum_{i=1}^{n-1} (K_{1-\alpha}^{n,i+1} - K_{1-\alpha}^{n,i})|t_n-t_i|\vertiii{\theta^i}_n^{2}+ \frac{L}{2\beta_0}K_{1-\alpha}^{n,1}|t_n-t_0|\vertiii{\theta^0}_n^{2}\\
&\leq \frac{1}{4}(|\lambda|\beta_2\|Ee_h^n\| + \|b\|_{\infty} \|\nabla e_h^n\|+ \|c\|_{\infty} \| e_h^n\|)^2 + \|\mathcal{L}_{0h}^n \theta^n\|^2 + \gamma_0\|P_h (\Upsilon^n + r^n -D^{\alpha}_{t_n}\eta^n)\|_1 \|\theta^n\|_1\\
&+ \frac{L}{2\beta_0}\sum_{i=1}^{n-1} (K_{1-\alpha}^{n,i+1} - K_{1-\alpha}^{n,i})|t_n-t_i|\vertiii{\theta^i}_n^{2}+ \frac{L}{2\beta_0}K_{1-\alpha}^{n,1}|t_n-t_0|\vertiii{\theta^0}_n^{2}.
\end{align*}
Thus, using AM-GM inequality, $H^1$-stability (\ref{l2projectionh1stability}) of $L^2$-projection, and the equivalent norms (\ref{equivalentnorms}), we arrive at
\begin{align*}
D_{t_n}^{\alpha}\vertiii{\theta^n}_n^2&\leq (|\lambda|\beta_2\|Ee_h^n\|  + \|b\|_{\infty} \|\nabla \eta^n\|+ \|c\|_{\infty} \| \eta^n\|)^2  +  (\|b\|_{\infty} + \|c\|_{\infty})^2 \| \theta^n\|_1^2 \\
&+ 2 \gamma_0\beta_3\| \Upsilon^n + r^n -D^{\alpha}_{t_n}\eta^n\|_1 \|\theta^n\|_1  + \frac{L}{\beta_0}\sum_{i=1}^{n-1} (K_{1-\alpha}^{n,i+1} - K_{1-\alpha}^{n,i})|t_n-t_i|\vertiii{\theta^i}_n^{2}+ \frac{L}{\beta_0}K_{1-\alpha}^{n,1}|t_n-t_0|\vertiii{\theta^0}_n^{2}\\
&\leq (|\lambda|\beta_2\|Ee_h^n\| +  \|b\|_{\infty} \|\eta^n\|_1+ \|c\|_{\infty} \| \eta^n\|)^2  +  \left(\frac{(\|b\|_{\infty}+\|c\|_{\infty})^2 }{\beta_0} \right)\vertiii{\theta^n}_{n}^2\\
&+ \frac{2 \gamma_0\beta_3}{\sqrt{\beta_0}}\| \Upsilon^n + r^n -D^{\alpha}_{t_n}\eta^n \|_1 \vertiii{\theta^n}_{n}+ \frac{L}{\beta_0}\sum_{i=1}^{n-1} (K_{1-\alpha}^{n,i+1} - K_{1-\alpha}^{n,i})|t_n-t_i|\vertiii{\theta^i}_n^{2}+ \frac{L}{\beta_0}K_{1-\alpha}^{n,1}|t_n-t_0|\vertiii{\theta^0}_n^{2}.
\end{align*}
Now, by applying the general Discrete Fractional Gr\"{o}nwall Inequality (Theorem \ref{DFGI}), we obtain
\begin{align*}
\vertiii{\theta^n}_n &\leq
2E_{\alpha}(2\Lambda_1 t_n^{\alpha})\Bigg(\vertiii{\theta^0}_n +  \frac{2\gamma_0\beta_3}{\sqrt{\beta_0}}\max_{1\leq j \leq n}\sum_{i=1}^jP_{\alpha}^{j,i}\|\Upsilon^i + r^i -D^{\alpha}_{t_n}\eta^i\|_1 \\
&+ \sqrt{2t_n^{\alpha}}\max_{1 \leq j \leq n }(\|b\|_{\infty} \|\eta^j\|_1+ \|c\|_{\infty} \| \eta^j\| +|\lambda| \beta_2\|Ee_h^j\|)\Bigg),
\end{align*}
where $\Lambda_{1}=\frac{(\|b\|_{\infty}+\|c\|_{\infty})^2 }{\beta_0}  + \frac{LT^{1-\alpha}}{\beta_0\Gamma(2-\alpha)} $. Finally, an application of the equivalent norms (\ref{equivalentnorms}) yields  
\begin{align*}
&\|\theta^n\|_1 \leq \frac{2}{\sqrt{\beta_0}}E_{\alpha}(2\Lambda_1 t_n^{\alpha})\Bigg(\sqrt{\gamma_0}\|\theta^0\|_1 + \frac{2\gamma_0\beta_3}{\sqrt{\beta_0}}\max_{1\leq j \leq n}\sum_{i=1}^jP_{\alpha}^{j,i}\left(\|\Upsilon^i \|_1+ \|r^i\|_1 +\|D^{\alpha}_{t_i}\eta^i\|_1\right)  \\
&\hspace{1cm}+ \sqrt{2t_n^{\alpha}}\max_{1 \leq j \leq n }\left(\|b\|_{\infty} \|\eta^j\|_1+ \|c\|_{\infty} \| \eta^j\| + |\lambda|  \beta_2\|Ee_h^j\|\right)\Bigg)\\
&\leq 2D_1 E_{\alpha}(2\Lambda_1 t_n^{\alpha})\Big(\|\theta^0\|_1 + \max_{1\leq j \leq n}\sum_{i=1}^jP_{\alpha}^{j,i}\left(\|\Upsilon^i\|_1 + \|r^i  \|_1\right) + \max_{1\leq j \leq n}\sum_{i=1}^jP_{\alpha}^{j,i}\|D^{\alpha}_{t_i}\eta^i\|_1 + \max_{1\leq j \leq n}\|b\|_{\infty} \|\eta^j\|_1\\
&+ \max_{1\leq j \leq n} \| \eta^j\| +\max_{1\leq j \leq n}\|Ee_h^j\| \Big),
\end{align*}
where $D_1=\max\left\{\sqrt{\frac{\gamma_0}{\beta_0}},\; \frac{2\gamma_0\beta_3}{\beta_0}, \;\frac{1}{\sqrt{\beta_0}}\sqrt{2T^{\alpha}}, \;\frac{\|c\|_{\infty}}{\sqrt{\beta_0}}\sqrt{2T^{\alpha}},\;\frac{|\lambda| \beta_2}{\sqrt{\beta_0}}\sqrt{2T^{\alpha}} \right\}.$ 
\end{proof}


\begin{lemma}\label{truncationerror} 
If the grading parameter $\gamma$ satisfies $1\;\leq\;\gamma\;\leq\;\frac{2(2-\alpha)}{\alpha} $ then under the assumptions in Theorem \ref{Regularity condition}, there holds
\begin{enumerate}
    \item[(i)] $\sum_{i=1}^nP_{\alpha}^{n,i}\|\Upsilon^i\|_m  \leq C \;\log_e (N) N^{-\min\{2-\alpha,\;\gamma\sigma_m \}}, \quad 1\leq n \leq N,$
    \item[(ii)] $\sum_{i=1}^nP_{\alpha}^{n,i}\|r^i\|_m  \leq C \;\log_e (N) N^{-\min\{2-\alpha,\;\gamma\sigma_m \}}, \quad 1\leq n \leq N,$
\end{enumerate}
where $\sigma_m=\alpha/(m+1)$, $m=0,1$, and $C$ is a positive constant which remains bounded as $\alpha\to 1^{-}$.
\end{lemma}

\begin{proof}
Under the assumptions in Theorem~\ref{Regularity condition}, using Lemma~5.2 in \cite{MR3639581} and Remark~5.5 in \cite{MR3639581} yields the following estimate
\begin{align}\label{trunc1}
    \|\Upsilon^n\| &\leq C \;n^{-\min\{ 2-\alpha,\; \gamma \alpha\}} = C\; t_n^{-\min\{\frac{2-\alpha}{\gamma}, \;\alpha \}} T^{\min\{\frac{2-\alpha}{\gamma}, \;\alpha \}} N^{-\min\{2-\alpha,\; \gamma\alpha \}}
    \end{align}
    and
    \begin{align}\label{trunc2}
    \|\Upsilon^n\|_1 &\leq C\; N^{\frac{\gamma \alpha}{2}}n^{-\min\{ 2-\alpha+ \frac{\gamma\alpha}{2}, \; \gamma \alpha\}} =  C \;t_n^{-\min\{\frac{2-\alpha}{\gamma}+\frac{\alpha}{2}, \;\alpha \}} T^{\min\{\frac{2-\alpha}{\gamma}+\frac{\alpha}{2},\; \alpha \}} N^{-\min\{2-\alpha, \;\frac{\gamma\alpha}{2} \}}
\end{align}
respectively, where the positive constant $C$ remains bounded as $\alpha \rightarrow 1^-$.

Define $\delta_N:=\frac{1}{8\log_e (N)}$, and $\beta_N^m:=\delta_N + \alpha - \min\left\{\frac{2-\alpha}{\gamma}+(\alpha-\sigma_m),\;\alpha\right\}$, and then use estimates (\ref{trunc1}) and (\ref{trunc2}) to obtain 

\begin{align*}
    \|\Upsilon^n\|_m &\leq  C_m\; t_n^{\beta_N^m - \alpha}\;t_n^{-\delta_N}\; N^{-\min\{2-\alpha,\; \gamma\sigma_m \}},\; C_m:= CT^{\min{\left(\frac{2-\alpha}{\gamma}+(\alpha-\sigma_m),\;\alpha\right)}}\\
    &\leq C_m \;t_n^{\beta_N^m - \alpha}\;t_1^{-\delta_N}\;  N^{-\min\{2-\alpha,\; \gamma\sigma_m \}}, \;m=0,1.
    \end{align*}
As for $1\;\leq\;\gamma\;\leq\;\frac{2(2-\alpha)}{\alpha} $ and $N>1$, $\beta_N^m\in(0,\;1)$, an appeal to Lemma~\ref{discretekernelproperties} $(e)$ yields 
\begin{align*}
    \sum_{j=1}^n P_{\alpha}^{n,j}\|\Upsilon^j\|_m &\leq C_m \;\frac{\Gamma(1+\beta_N^m -\alpha)}{\Gamma(1+\beta_N^m)} t_n^{\beta_N^m} \;t_1^{-\delta_N}\; N^{-\min\{2-\alpha,\; \gamma\sigma_m \}}\\
    &\leq C_m \;\frac{\Gamma(1+\beta_N^m -\alpha)}{\Gamma(1+\beta_N^m)} t_n^{\delta_N} \;t_1^{-\delta_N}\;  N^{-\min\{2-\alpha,\; \gamma\sigma_m \}}\\
    &\leq C_m \;\frac{\Gamma(1+\beta_N^m -\alpha)}{\Gamma(1+\beta_N^m)} N^{\gamma\delta_N}\;  N^{-\min\{2-\alpha,\; \gamma\sigma_m \}}\;=\; C_m \;\frac{\Gamma(1+\beta_N^m -\alpha)}{\Gamma(1+\beta_N^m)} e^{\gamma/8}\; N^{-\min\{2-\alpha,\; \gamma\sigma_m \}}\\
    &\leq C_m\; 2\Gamma(1+\delta_N -\alpha) e^{\gamma/8}\;  N^{-\min\{2-\alpha,\; \gamma\sigma_m \}},\; m=0,1.
\end{align*}
Thus, the limit $\lim_{\alpha\to 1^{-}}\Gamma(1+\delta_N - \alpha) = \frac{\Gamma(1+\delta_N)}{\delta_N}$ and the above estimate yields the estimate $(i)$. Following a similar argument, one can derive the second estimate $(ii)$.
\end{proof}
\begin{lemma}\label{DiscreteDerivativeDifference}
For $\phi^n\in H^m(\Omega),\;0\leq n\leq N,\;m=0,1,$ there holds
\begin{align*}
    & \sqrt{\sum_{j=1}^n P^{n,j}_{\alpha}\|D_{t_j}^{\alpha} (t_j \phi^j) - t_j D_{t_j}^{\alpha} \phi^j\|_m^2 } \;\leq\; C\; t_n^{1-\frac{\alpha}{2}}\max_{0\leq j \leq n}\|\phi^j\|_m,\quad 1\leq n\leq N,
\end{align*}
where the positive constant $C$ remains bounded as $\alpha\to 1^{-}$.
\end{lemma}
\begin{proof}
\noindent Apply the definition (\ref{L1scheme}) of discrete fractional derivative $D_{t_n}^{\alpha}$ and summation by parts formula $\sum_{j=2}^n f_j(g_j -g_{j-1}) = - \sum_{j=2}^n  (f_j - f_{j-1})g_{j-1} + f_n g_n - f_1 g_1$ to get the following result
\begin{align}
     \nonumber &D_{t_n}^{\alpha} (t_n \phi^n) - t_n D_{t_n}^{\alpha} \phi^n =   \sum_{j=2}^n (t_n-t_{j-1})\Big(K_{1-\alpha}^{n,j} - K_{1-\alpha}^{n,j-1} \Big)\phi^{j-1}+ (t_n - t_0)K^{n,1}_{1-\alpha}\phi^0 . 
\end{align}
Now, using the inequality (\ref{pollutionbound}) we obtain
\begin{align*}
    \|D_{t_n}^{\alpha} (t_n \phi^n) - t_n D_{t_n}^{\alpha} \phi^n \|_m\;& \leq  \; \max_{0\leq j \leq n}\|\phi^j\|_m \left(  \sum_{j=2}^n (t_n-t_{j-1})\left(K_{1-\alpha}^{n,j} - K_{1-\alpha}^{n,j-1} \right) + (t_n - t_0)K^{n,1}_{1-\alpha} \right)  \\
   &\leq\; \max_{0\leq j \leq n}\|\phi^j\|_m \frac{t_n^{1-\alpha}}{\Gamma{(2-\alpha)}}.
\end{align*}
Finally, an appeal to triangle inequality and Lemma~\ref{discretekernelproperties} $(c)$ with $j= 1$ yields the desired result.
\end{proof}

\begin{lemma}\label{etaestimate}
Under the assumptions in Theorem \ref{Regularity condition}, the following estimate holds-
\begin{align*}
    \sum_{j=1}^nP_{\alpha}^{n,j}\|D^{\alpha}_{t_j} \eta^j\|_m \leq C\left( h^{2-m}\log_e (N) + N^{-\min{\{2-\alpha,\;\gamma\sigma_m\}}} \right), \quad 1\leq n \leq N,
\end{align*}
where $\sigma_m=\alpha/(m+1),\;m=0,1,$ and the positive constant $C$ remains bounded as $\alpha\to 1^{-}$.
\end{lemma}
\begin{proof} 
Consider
\begin{align}\label{alpharobustetaestimate}
   \nonumber \sum_{j=1}^n P_{\alpha}^{n,j}\|D^{\alpha}_{t_j}\eta^j\|_m &\leq \sum_{j=1}^n P_{\alpha}^{n,j} \sum_{i=1}^j K_{1-\alpha}^{j,i} \| \eta^i - \eta^{i-1} \|_m \leq \sum_{i=1}^n \sum_{j=i}^n P_{\alpha}^{n,j} K_{1-\alpha}^{j,i} \| \int_{t_{i-1}}^{t_i} \partial_s \eta(s)ds \|_m\\
    & \leq \sum_{i=1}^n   \int_{t_{i-1}}^{t_i} \|\partial_s \eta(s) \|_m ds,
\end{align}
where we have used Lemma~\ref{discretekernelproperties}
 $(b)$ to obtain the last inequality. As under the assumptions in Theorem~\ref{Regularity condition}, $\|s\partial_s u(s)\|_2 \leq C, \; s \in(0,T]$, we derive the required estimate using (\ref{alpharobustetaestimate}) and the approximation property (\ref{optimaletaestimate2}) of the elliptic projection as follows
 \begin{align*}
  \sum_{j=1}^n P_{\alpha}^{n,j}\|D^{\alpha}_{t_j}\eta^j\|_m &\leq \int_0^{t_1}\|\partial_s \eta(s)\|_m ds + \sum_{i=2}^n \int_{t_{i-1}}^{t_i} \|\partial_s \eta(s)\|_m ds\\
   &\leq C h^{1-m}\int_0^{t_1}\|\partial_s u(s)\|_1 ds + Ch^{2-m}\sum_{i=2}^n \int_{t_{i-1}}^{t_i} s^{-1} ds\\ 
   &\leq C h^{1-m}\int_0^{t_1}s^{\frac{\alpha}{2}-1} ds + Ch^{2-m}\log_e \left(\frac{t_n}{t_1}\right)\\
   &\leq C h^{1-m} \frac{2t_1^{\frac{\alpha}{2}}}{\alpha} + Ch^{2-m}\log_e\left( n\right)\\
   & \leq C\max\left\{2,\; \frac{2 T^{\alpha/2}}{\alpha},\; \frac{T^\alpha}{\alpha^2}\right\} \left( N^{-\gamma \sigma_m} + h^{2-m} \log_e (n)\right).
 \end{align*}
\end{proof}
\begin{remark}\label{remark4.1}
    Under an additional regularity assumption $\|\partial_t u(t)\|_2\;\leq\; C\;t^{\sigma-1},$ for some $\sigma>0,$ the following estimate can be obtained
\begin{align*}
\sum_{j=1}^nP_{\alpha}^{n,j}\|D^{\alpha}_{t_j} \eta^j\|_m \leq C\; h^{2-m}, \quad 1\leq n \leq N,
\end{align*}
which is independent of $\log_e (N)$ factor.
\end{remark}
Now, our main result is established in the following Theorem.

\begin{theorem}\label{L2H1errortheorem}
Let $u_h^n$ and $u(t_{n})$ be the solution of the problem (\ref{fullydiscrete}) and (\ref{variation1}) at the temporal grid $t_{n}$, respectively. Then, under the assumptions in Theorem \ref{Regularity condition}, the following estimate holds
\begin{align}
\nonumber& \max_{1\leq n \leq N} \|u_h^n - u(t_{n}) \|_m \leq   C\; \log_e (N)\left(h^{2-m}+  N^{-\min\{\gamma\sigma_m, 2-\alpha\}}   \right),\;m=0,1, 
\end{align}
where $\sigma_m=\alpha/(m+1)$. In particular, when the grading parameter $\displaystyle \gamma = \frac{2(2-\alpha)}{\alpha}$, there holds
\begin{align}
\nonumber& \max_{1\leq n \leq N} \|u_h^n - u(t_{n}) \|_m \leq   C \; \log_e (N) (h^{2-m}  + N^{- (2-\alpha)}   ),\;m=0,1,
\end{align}
where the positive constants $C$ appearing in the above estimates remains bounded as $\alpha\to 1^{-}$.
\end{theorem}
\begin{proof}
An application of the estimate (\ref{optimaletaestimate}),  Lemma~\ref{truncationerror} and Lemma~\ref{etaestimate} in Lemma~\ref{thetaestimate_l2norm} yields
\begin{align*}
     \|u_h^n - u(t_{n}) \| \;&\leq\; \|\eta^n\| + \|\theta^n\| \; \leq\;   C \;\log_e( N)(h^2 + N^{-\min\{\gamma\alpha, 2-\alpha\}}  ). 
\end{align*}
Now, by applying the above estimate, the estimate (\ref{optimaletaestimate}), Lemma~\ref{truncationerror} and Lemma~\ref{etaestimate} in Lemma~\ref{thetaestimate_h1norm}, we obtain
\begin{align*}
     \|u_h^n - u(t_{n}) \|_1 \;&\leq\; \|\eta^n\|_1 + \|\theta^n\|_1 \;\leq\;  C \;\log_e (N)(h + N^{-\min\{\frac{\gamma\alpha}{2}, 2-\alpha\}}   ).
\end{align*}
Hence, the desired result follows by combining the previous two estimates.  
\end{proof}
\section{$L^{\infty}$-norm estimate }\label{Linftyerroranalysis}
When the convection coefficient $\boldsymbol{b}$ in (\ref{EllipticOperator}) is zero, a superconvergence in $H^1$-norm error estimate of $\theta^n$ is derived, and as a consequence, an $L^{\infty}$ error estimate established for the 2D-problem (\ref{pide}) in this section. To obtain these results, we first establish a few auxiliary results. 
\begin{lemma}\label{thetaestimate_h1norm_sup}
Under the assumptions in Theorem \ref{Regularity condition} and maximum temporal grid size restriction (\ref{timecondition}), there holds
 \begin{align}
 \nonumber& \|t_n\theta^n\|_1\; \leq 2D_2 E_{\alpha}(2\widetilde{\Lambda}_2 t_n^{\alpha})\;\max_{1\leq j \leq n}\Bigg(\sum_{i=1}^jP_{\alpha}^{j,i}\left(\|t_i \Upsilon^i \|_1 +\| t_i r^i  \|_1\right)+ \sqrt{\sum_{i=1}^{j}P_{\alpha}^{j,i} \left(\|D^{\alpha}_{t_i}(t_i \eta^i)\|^2 + \|T^i\|^2\right)} \\
\nonumber&\hspace{7cm}  +\|b\|_{\infty} \|t_j \eta^j\|_1+ \|t_j e_h^j\|+\|E(t_j e_h^j)\|\quad \Bigg),\;1\leq n \leq N,
\end{align}
where $T^j:= D_{t_j}^{\alpha} (t_j e_h^j) - t_j D_{t_j}^{\alpha} e_h^j, \; D_2=\frac{2}{\sqrt{\beta_0}}\max\left\{ \frac{\gamma_0\beta_3}{\sqrt{\beta_0}},\;1,\;|\lambda|\beta_2\sqrt{T^{\alpha}}, \;\sqrt{T^{\alpha}},  \;\|c\|_{\infty}\sqrt{T^{\alpha}}\right\}$ and $\widetilde{\Lambda}_2=\frac{\|b\|_{\infty}^2 }{\beta_0} + \frac{LT^{1-\alpha}}{\beta_0\Gamma(2-\alpha)}$.
\end{lemma}
\begin{proof}
Multiply $t_n$ in $(\ref{erroreqtheta})$ and then set $v_h = \mathcal{L}^n_{0h}t_n \theta^n$ to obtain
\begin{align}
    \nonumber (D^\alpha_{t_n}(t_n \theta^n), \mathcal{L}^n_{0h}(t_n \theta^n)) + a_0(t_n; t_n \theta^n, \mathcal{L}^n_{0h}(t_n \theta^n) )= &\; (\lambda\mathcal{I}(E(t_n e_h^n)), \mathcal{L}^n_{0h}(t_n \theta^n)) -a_1(t_n; t_n e^n_h, \mathcal{L}^n_{0h}(t_n \theta^n))\\
    \label{H1erroreq1}& + (t_n \Upsilon^n + t_n r^n -D^\alpha_{t_n}(t_n\eta^n) + T^n, \mathcal{L}^n_{0h}(t_n \theta^n)),
\end{align}
where $T^n:= D_{t_n}^{\alpha} (t_n e_h^n) - t_n D_{t_n}^{\alpha} e_h^n$. Apply the estimates (\ref{deoestimate}), (\ref{pollutionterm}), and the $L^2$-projection $P_h$ in (\ref{H1erroreq1}) to obtain
\begin{align*}
    &\frac{1}{2}D^\alpha_{t_n}\vertiii{t_n \theta^n}_n^2 + \|\mathcal{L}^n_{0h}(t_n \theta^n)\|^2 \leq \; \frac{1}{2}\sum_{j=1}^{n-1}(K_{1-\alpha}^{n,j+1} - K_{1-\alpha}^{n,j})(\vertiii{t_j \theta^j}_n^2 - \vertiii{t_j \theta^j}_j^2) + (P_h (t_n \Upsilon^n + t_n r^n) , \mathcal{L}^n_{0h}(t_n \theta^n))  \\
    &\hspace{3cm }+ (\lambda\mathcal{I}(E(t_n e_h^n))-D^\alpha_{t_n}(t_n\eta^n) + T^n, \mathcal{L}^n_{0h}(t_n \theta^n)) -a_1(t_n; t_n e^n_h, \mathcal{L}^n_{0h}(t_n \theta^n)),
\end{align*}
and then an application of the estimate (\ref{Lipschitz}), relation (\ref{deo}) and equivalence norms (\ref{equivalentnorms}) yields
\begin{align*}
 &\frac{1}{2}D^\alpha_{t_n}\vertiii{t_n \theta^n}_n^2 + \|\mathcal{L}^n_{0h}(t_n \theta^n)\|^2 \leq   \frac{L}{2\beta_0}\sum_{j=1}^{n-1}(K_{1-\alpha}^{n,j+1} - K_{1-\alpha}^{n,j})|t_n- t_j|\vertiii{t_j \theta^j}_n^2 + a_0(t_n;t_n \theta^n, P_h (t_n \Upsilon^n + t_n r^n) ) \\
 &\hspace{3cm }+ (\lambda\mathcal{I}(E(t_n e_h^n))-D^\alpha_{t_n}(t_n\eta^n) + T^n, \mathcal{L}^n_{0h}(t_n \theta^n)) -a_1(t_n; t_n e^n_h, \mathcal{L}^n_{0h}(t_n \theta^n)).
\end{align*}
Now, by using (\ref{bounded}), (\ref{a1bounded1}), (\ref{integralbounded}) and the Cauchy-Schwarz inequality, we obtain
\begin{align*}
& \frac{1}{2}D^\alpha_{t_n}\vertiii{t_n \theta^n}_n^2 + \|\mathcal{L}^n_{0h}(t_n \theta^n)\|^2 \leq  \frac{L}{2\beta_0}\sum_{j=1}^{n-1}(K_{1-\alpha}^{n,j+1} - K_{1-\alpha}^{n,j})|t_n- t_j|\vertiii{t_j \theta^j}_n^2  
+ \gamma_0\|P_h(t_n \Upsilon^n + t_n r^n)\|_1 \|t_n \theta^n\|_1\\
 & \hspace{3cm }+ \left(|\lambda|\beta_2\|E(t_n e_h^n)\| + \|D^\alpha_{t_n}(t_n \eta^n)\| + \|T^n\| + \|b\|_{\infty}\|\nabla (t_n e_h^n)\| +  \|c\|_{\infty}\|t_n e_h^n\| \right)\| \mathcal{L}^n_{0h}(t_n \theta^n)\|.
\end{align*}
An application of AM-GM inequality, $H^1$-norm stability (\ref{l2projectionh1stability}) of $P_h$ and the equivalent norm (\ref{equivalentnorms}) shows
\begin{align*}
& \frac{1}{2}D^\alpha_{t_n}\vertiii{t_n \theta^n}_n^2 + \|\mathcal{L}^n_{0h}(t_n \theta^n)\|^2 \leq  \frac{L}{2\beta_0}\sum_{j=1}^{n-1}(K_{1-\alpha}^{n,j+1} - K_{1-\alpha}^{n,j})|t_n- t_j|\vertiii{t_j \theta^j}_n^2  
+  \frac{\gamma_0\beta_3}{\sqrt{\beta_0}}\|t_n\Upsilon^n + t_nr^n\|_1 \vertiii{t_n \theta^n}_n \\  
 &\hspace{0cm }+\frac{1}{2}(|\lambda| \beta_2 \|E(t_n e_h^n)\| + \|D^\alpha_{t_n}(t_n\eta^n)\|+ \|T^n\| + \|b\|_{\infty}\|t_n\eta^n\|_1 + \|c\|_{\infty}\|t_n e_h^n\| )^2 + \frac{\|b\|_{\infty}^2}{2\beta_0}\vertiii{t_n \theta^n}_n^2 +  \|\mathcal{L}^n_{0h}(t_n \theta^n)\|^2.
\end{align*}
Thus, 
\begin{align*}
& D^\alpha_{t_n}\vertiii{t_n \theta^n}_n^2 \leq \frac{\|b\|_{\infty}^2}{\beta_0}\vertiii{t_n \theta^n}_n^2 + \frac{L}{\beta_0}\sum_{j=1}^{n-1}(K_{1-\alpha}^{n,j+1} - K_{1-\alpha}^{n,j})|t_n- t_j|\vertiii{t_j \theta^j}_n^2   +  \frac{2\gamma_0\beta_3}{\sqrt{\beta_0}}\|t_n\Upsilon^n + t_nr^n\|_1 \vertiii{t_n \theta^n}_n\\  
 &\hspace{3cm }+2\left(|\lambda| \beta_2 \|E(t_n e_h^n)\|  + \|b\|_{\infty}\|t_n\eta^n\|_1 + \|c\|_{\infty}\|t_n e_h^n\| \right)^2 + 2\left(\|D^\alpha_{t_n}(t_n\eta^n)\| + \|T^n\|\right)^2
\end{align*}
and hence, by applying the general fractional Gronwall's Inequality (Theorem~\ref{DFGI}), we obtain
\begin{align*}
&\vertiii{t_n \theta^n}_n \leq \; 2E_{\alpha}(2 \Lambda_2 t_n^\alpha) \Bigg( \frac{2\gamma_0\beta_3}{\sqrt{\beta_0}}\max_{1\leq j \leq n}\sum_{i=1}^jP_\alpha^{j,i}(\|t_i \Upsilon^i\|_1 +\|t_i r^i\|_1)+2\max_{1\leq j \leq n} \sqrt{\sum_{i=1}^{j}P_{\alpha}^{j,i} \left(\|D^{\alpha}_{t_i}(t_i \eta^i)\|^2 + \|T^i\|^2\right)}\\
&\hspace{3cm }+ 2\sqrt{t_n^\alpha}\max_{1\leq j \leq n}\left(|\lambda|\beta_2\|E(t_je_h^j)\| + \|b\|_{\infty}\|t_j \eta^j\|_1 + \|c\|_{\infty}\|t_j e_h^j\|\right) \Bigg). 
\end{align*}
Finally, an application of equivalent norms (\ref{equivalentnorms}) yields the result.
\end{proof} 

\begin{lemma}\label{etaestimatemoreregular}
Under the assumptions in Theorem \ref{Regularity condition}, the following estimate holds
\begin{align}
\nonumber&\sqrt{\sum_{j=1}^nP_{\alpha}^{n,j}\|D^{\alpha}_{t_j} (t_j\eta^j)\|^2} \leq C\; h^{2}\;t_n^{1-\frac{\alpha}{2}} , \quad 1\leq n \leq N.
\end{align}
\end{lemma}
\begin{proof}
An appeal to the approximation property (\ref{optimaletaestimate2}) of the elliptic projection $\mathcal{R}_{h}(t): H_{0}^{1}(\Omega)\to S_{h}, \; t\in [0,T], $ and the regularity result Theorem \ref{Regularity condition} yields
\begin{align}
\nonumber \|D^{\alpha}_{t_n} (t_n \eta^n)\| \;& \leq\; \sum_{j=1}^{n}K_{1-\alpha}^{n,j}\left\|(t_j\eta^j - t_{j-1}\eta^{j-1})\right\| \leq\; \sum_{j=1}^n K_{1-\alpha}^{n,j}\int_{t_{j-1}}^{t_j}\left\|\partial_s (s\eta(s))\right\| ds\\
\nonumber & \leq\; Ch^{2}\;\sum_{j=1}^n K_{1-\alpha}^{n,j}\;\Delta t_j\;=\; Ch^{2}\frac{t_n^{1-\alpha}}{\Gamma{(2-\alpha)}},\;1\leq n\leq N,
\end{align}  
Now, apply Lemma~\ref{discretekernelproperties} $(c)$ with $j=1$ and the above estimate to get the desired estimate.
\end{proof}
Finally, a sharp estimate of $\theta^n:=\mathcal{R}_h(t_n)u(t_n)-u^n_h $ in $H^1$-norm is obtained in the following result.
\begin{lemma}\label{superconvergenceregular}
Let the convection coefficient $\boldsymbol{b}$ in (\ref{EllipticOperator}) be zero.
Then, under the assumptions in Theorem \ref{Regularity condition}, there holds
\begin{align*}
& \max_{1\leq n \leq N}\;t_n^{\frac{\alpha}{2}} \| \theta^n \|_1 \; \leq \;  C\;\log_e(N) \left(h^{2}+ N^{-\min\{\frac{\gamma\alpha}{2}, 2-\alpha\}}  \right).
\end{align*}
In particular, when the grading parameter $\displaystyle \gamma = \frac{2(2-\alpha)}{\alpha}$, there holds
\begin{align*}
& \max_{1\leq n \leq N} \;t_n^{\frac{\alpha}{2}} \| \theta^n \|_1 \; \leq\; C\;\log_e(N) \left(h^{2} + N^{-( 2-\alpha)}  \right).
\end{align*}
\end{lemma}
\begin{proof}
An application of Lemma~\ref{truncationerror}, Lemma~\ref{DiscreteDerivativeDifference}, Theorem~\ref{L2H1errortheorem} and Lemma~\ref{etaestimatemoreregular} yields
\begin{align*}
    \sum_{i=1}^j P_\alpha^{j,i}\left(\|t_i \Upsilon^i\|_1 + \|t_i r^i\|_1\right) &\leq Ct_j \log_e(N)\;N^{-\min\{\frac{\gamma\alpha}{2}, 2-\alpha \}}, \quad 1 \leq j \leq N, \\
    \sqrt{\sum_{i=1}^j P_\alpha^{j,i} \|T^i\|^2} &\leq Ct_j^{1-\frac{\alpha}{2}}\max_{0\leq i \leq j}\|e_h^i\|^2, \quad 1\leq j \leq N, \\
    \max_{0\leq i \leq j}\|e_h^i\|^2 & \leq C\;\log_e(N)\left(h^2 + N^{-\min\{\gamma\alpha, 2-\alpha \}}  \right), \\
    \sqrt{\sum_{i=1}^j P_\alpha^{j,i} \|D^\alpha_{t_i}(t_i \eta^i)\|^2} &\leq C h^2 t_j^{1-\frac{\alpha}{2}}, \quad 1 \leq j \leq N.
\end{align*}
Now, apply the above estimate in Lemma~\ref{thetaestimate_h1norm_sup} to get the desired result.
\end{proof}
As a consequence of the super-convergence estimate for $\theta^n$ in Lemma~\ref{superconvergenceregular}, we obtain the following $L^{\infty}$-norm estimate in one and two dimensions.

\begin{theorem}\label{Linftynormestimate}
 Let $u_h^n$ and $u(t_{n})$ be the solution of the problem (\ref{fullydiscrete}) and (\ref{variation1}) at the temporal grid $t_{n}$, respectively. Further, let the convection coefficient $\boldsymbol{b}$ in (\ref{EllipticOperator}) be zero. Then, under the assumptions in Theorem \ref{Regularity condition}, $\displaystyle \gamma = \frac{2(2-\alpha)}{\alpha}$, and quasi-uniform triangulation $\mathcal{T}_h$, there exists a positive constant $C$ independent of $h$, $N$ and $p$ such that
\begin{align*}
  t_n^{\frac{\alpha}{2}}\|u_h^n - u(t_n)\|_{L^{\infty}(\Omega)}\; \leq\; 
  C \ell_{h,d} \;\log_e (N)\left( h^{2-\frac{2}{p}}\|t_n^{\frac{\alpha}{2}}\;u(t_n)\|_{W^{2,p}(\Omega)} + N^{ -(2-\alpha)} \right), \; 1\leq n \leq N,\; p\geq d,
 \end{align*}
 where the constant 
$C$ remains bounded as $\alpha\to 1^{-}$, and $\ell_{h,d}:=\begin{cases}
     1 & {:} \; d=1,\\
     1+|\log_e {h}| &{:}\; d=2.
 \end{cases}$
\end{theorem}
\begin{proof}
Using Sobolev inequality, inverse inequality (see \cite{MR2249024} Lemma~6.4) and Lemma~\ref{superconvergenceregular}, we obtain
 \begin{align*}
     t_n^{\frac{\alpha}{2}}\|\theta^n\|_{L^{\infty}(\Omega)}\leq C\;\ell_{h,d}^{1/2} \|\nabla\theta^n\| \leq C\;\ell_{h,d}^{1/2}\;\log_e (N)(h^{2}+ N^{ -(2-\alpha)}  ),
 \end{align*}
and, from \cite{MR2249024}  (equation (6.81) on page 103), we obtain
 \begin{align*} \|\eta^n\|_{L^{\infty}(\Omega)}\leq  C \ell_{h,d}\; h^{2-\frac{2}{p}}\|u(t_n)\|_{W^{2,p}(\Omega)}.
 \end{align*}
Now, using $\|u_h^n - u(t_n)\|_{L^{\infty}(\Omega)} \leq \|\theta^n\|_{L^{\infty}(\Omega)}+\|\eta^n\|_{L^{\infty}(\Omega)}$ and the above two estimates we get the desired result.
\end{proof}


\section{Numerical results} \label{section6}  
This section provides numerical experiments to justify our theoretical findings. Let us take $h^2 = N^{-(2-\alpha)}$ in Theorem~\ref{L2H1errortheorem} and Theorem~\ref{Linftynormestimate}, respectively, and then compute the rate of convergence with respect to $H^m$-norm using the formula $R^m = \log_e\left(\frac{E_{h_1}^m}{E_{h_2}^m}\right)/\log_e\left(\frac{h_1}{h_2}\right),\;m=0,1$, and with respect to max-norm via $R^\infty = \log_e\left(\frac{E_{h_1}^\infty}{E_{h_2}^\infty}\right)/\log_e\left(\frac{h_1}{h_2}\right)$, where 
\begin{align}
    \nonumber & E_{h}^m:= \begin{cases}
        \max\limits_{1 \leq n \leq N}\|u_h^n - u(t_n)\|_m & \text{: if the analytical solution $u$ is known,}\\
        \max\limits_{1 \leq n \leq N}\|u_{h}^n - \widetilde{u}_h(t_n)\|_m  & \text{: otherwise, }
    \end{cases} \\
    \nonumber & E_{h}^\infty:= \begin{cases}
        \max\limits_{\boldsymbol{x}_j\in\mathcal{N},\;1 \leq n \leq N}\;t_n^{\frac{\alpha}{2}}|u_h^n(\boldsymbol{x}_j) - u(\boldsymbol{x}_j,\;t_n)|& \text{: if the analytical solution $u$ is known,}\\
        \max\limits_{\boldsymbol{x}_j\in\mathcal{N},\;1 \leq n \leq N}\;t_n^{\frac{\alpha}{2}}|u_{h}^n(\boldsymbol{x}_j) - \widetilde{u}_h(\boldsymbol{x}_j,\; t_n)| & \text{: otherwise, }
    \end{cases}
\end{align}

    \noindent $\mathcal{N}:=\{\boldsymbol{x}_j\}_{j=1}^{N_h + 500} $ is a collection of points in the triangulation $\mathcal{T}_h\cap\overline{\Omega}\subseteq \mathbb{R}^d,\;d=1,2,$ $N_h$ denotes the number of nodal points in $\mathcal{T}_h\cap\overline{\Omega}$, and $\widetilde{u}_h(t),\;t\in[0,\;T]$, is obtained from $ u^k_{\frac{h}{2}},\;0\leq k \leq 2N$, by applying a piece-wise linear interpolation in time and a piece-wise quadratic interpolation in space direction. 

As the solution is non-smooth in time for all the cases, we have used graded mesh with the grading parameter $\gamma = 2\frac{(2-\alpha)}{\alpha}$ to resolve the initial singularity. In each example, the errors $E_h^p$ and the corresponding computed rate of convergence $R^p,\;p=0,1,\infty,$ are displayed for $\alpha = 0.2, \;0.5$ and $ 0.8$. The implementation is conducted using the FreeFem++ software.

\begin{example}\label{1Dpde1}  
In this example, time-fractional PDE in one dimension is considered with variable coefficients. For $\Omega = (0,1)$, let the exact solution to the problem (\ref{pide}) is given as $u(x,t) = \sin(\pi x)(t^{\alpha}+t^3)$
with $A(x,t)=2+ x^2 + \sin(t),\;
 b(x,t) =1 + x^2 + t^2,$ $c(x,t) =1 + 2x^2 + \sin(t)$, $\lambda =0$, and $T=1$. The initial condition $u_0$ and the source term $f$ are chosen according to the exact solution $u$.
 The computed rate of convergence (ROC) $R^0$, $R^1$, and $R^\infty$ are listed in Table~\ref{L2tableforpde1freefemf4}. It is observed that the computed rate of convergence confirms the theoretical rate of convergence.

\begin{table}[H]
\centering
\begin{tabular}{||l|l|l|l|l|l|l||}
\hline
\multicolumn{2}{||l|}{$N$} & 4 & 8 & 16 & 32 & 64 \\ \hline \hline
\multicolumn{1}{||l|}{\multirow{6}{*}{$\alpha$=0.2}} & $E^0_h$ & 0.0624828
& 0.0185501
& 0.00424618
& 0.00107117
& 0.000289391
\\ \cline{2-7} 
\multicolumn{1}{||l|}{}                  & $R^0$ & - & 2.17
& 2.38
& 2.41
& 2.09
 \\ \cline{2-7} 
\multicolumn{1}{||l|}{}                  & $E^1_h$ & 1.00104
& 0.574667
& 0.309932
& 0.175225
& 0.0937208
\\ \cline{2-7} 
\multicolumn{1}{||l|}{}                  & $R^1$ & - & 0.99
& 1.00
& 1.00
& 1.00
\\ \cline{2-7} 
\multicolumn{1}{||l|}{}                  & $E^\infty_h$ & 0.119643
& 0.038039
& 0.00890195
& 0.00251029
& 0.000954948
\\ \cline{2-7} 
\multicolumn{1}{||l|}{}                  & $R^\infty$ & - & 2.05
& 2.35
& 2.22
& 1.54
  \\ \hline \hline
\multicolumn{1}{||l|}{\multirow{6}{*}{$\alpha$=0.5}} & $E^0_h$ & 0.106965
& 0.0357494
& 0.0132082
& 0.00399679
& 0.00144161
\\ \cline{2-7} 
\multicolumn{1}{||l|}{}                  & $R^0$ &  - & 2.15
& 2.12
& 2.14
& 2.05
\\ \cline{2-7} 
\multicolumn{1}{||l|}{}   & $E^1_h$ & 1.32926
& 0.803416
& 0.503175
& 0.287748
& 0.175175
   \\ \cline{2-7} 
\multicolumn{1}{||l|}{}                  & $R^1$ &  - & 0.99
& 1.00
& 1.00
& 1.00
\\ \cline{2-7} 
\multicolumn{1}{||l|}{}                  & $E^\infty_h$ & 0.217937
& 0.0735211
& 0.0272697
& 0.00837528
& 0.00302581
\\ \cline{2-7} 
\multicolumn{1}{||l|}{}                  & $R^\infty$ &  - & 2.13
& 2.11
& 2.11
& 2.05
\\ \hline \hline
\multicolumn{1}{||l|}{\multirow{6}{*}{$\alpha$=0.8}} & $E^0_h$ & 0.101002
& 0.0558257
& 0.0240716
& 0.0138215
& 0.00507639
\\ \cline{2-7} 
\multicolumn{1}{||l|}{}                  & $R^0$ & - & 2.06
& 2.07
& 1.93
& 2.06
 \\ \cline{2-7} 
\multicolumn{1}{||l|}{}                  & $E^1_h$ &  1.33136
& 1.0026
& 0.67028
& 0.503078
& 0.309822
 \\ \cline{2-7} 
\multicolumn{1}{||l|}{}                  & $R^1$ & - & 0.99
& 0.99
& 1.00
& 1.00
\\ \cline{2-7} 
\multicolumn{1}{||l|}{}                  & $E^\infty_h$ &  0.207607
& 0.109068
& 0.049018
& 0.0283475
& 0.0104925
   \\ \cline{2-7} 
\multicolumn{1}{||l|}{}                  & $R^\infty$ & - & 2.24
& 1.97
& 1.90
& 2.05
 \\ 
\hline 
\end{tabular}
\caption{Error $E_h^p$ and rate of convergence $R^p \; p=0,1,\infty$ of the proposed method for Example~\ref{1Dpde1}.}\label{L2tableforpde1freefemf4}
\end{table}

 \end{example}

\begin{example}\label{1Dpde2}  
In this example, one-dimensional time-fractional PDE (\ref{pide}) defined over the interval $\Omega = (-1,1)$ and $T=1$ with $A(x,t)=2 + x^2t,\;
 b(x,t) =xt,$ $c(x,t) =x^2t$, $\lambda =0$, and the exact solution 
$u(x,t)= (1+t^{\alpha})x(1- |x|)$. The initial condition and the source term are chosen accordingly. The computed rates of convergence $R^0$, $R^1$, and $R^\infty$ listed in Table~\ref{Linftableforpde2freefemf4} align with the theoretical convergence rate.

\begin{table}[H]
\centering
\begin{tabular}{||l|l|l|l|l|l|l||}
\hline 
\multicolumn{2}{||l|}{$N$} & 4 & 8 & 16 & 32 & 64  \\ \hline \hline
\multicolumn{1}{||l|}{\multirow{6}{*}{$\alpha$=0.2}} & $E^0_h$ & 0.0347015
& 0.0105464
& 0.00293698
& 0.00087132
& 0.000256965
\\ \cline{2-7} 
\multicolumn{1}{||l|}{}                  & $R^0$ & - & 1.92
& 1.95
& 1.99
& 1.99
 \\ \cline{2-7} 
\multicolumn{1}{||l|}{}                  & $E^1_h$ & 0.444256
& 0.244647
& 0.128841
& 0.0710062
& 0.0382664
 \\ \cline{2-7} 
\multicolumn{1}{||l|}{}                  & $R^1$ & - & 0.96
& 0.98
& 0.98
& 1.01
   \\ \cline{2-7} 
\multicolumn{1}{||l|}{}                  & $E^\infty_h$ & 0.0384505
& 0.011422
& 0.00316715
& 0.000933004
& 0.000282078
   \\ \cline{2-7} 
\multicolumn{1}{||l|}{}                  & $R^\infty$ & - & 1.96
& 1.96
& 2.00
& 1.95
  \\ \hline \hline
\multicolumn{1}{||l|}{\multirow{6}{*}{$\alpha$=0.5}} & $E^0_h$  & 0.0510937
& 0.0182744
& 0.00711651
& 0.00249264
& 0.000861023
   \\ \cline{2-7} 
\multicolumn{1}{||l|}{}                  & $R^0$ &  - & 2.01
& 2.01
& 2.00
& 2.00
 \\ \cline{2-7} 
\multicolumn{1}{||l|}{}   & $E^1_h$ & 0.547348
& 0.327248
& 0.204283
& 0.119419
& 0.0710064
   \\ \cline{2-7} 
\multicolumn{1}{||l|}{}                  & $R^1$ &  - & 1.01
& 1.00
& 1.03
& 0.98
 \\ \cline{2-7} 
\multicolumn{1}{||l|}{}                  & $E^\infty_h$ & 0.0518695
& 0.0190626
& 0.00756041
& 0.00271262
& 0.000932223
  \\ \cline{2-7} 
\multicolumn{1}{||l|}{}                  & $R^\infty$ &  - & 1.96
& 1.97
& 1.96
& 2.00
  \\ \hline \hline
\multicolumn{1}{||l|}{\multirow{6}{*}{$\alpha$=0.8}} & $E^0_h$ & 0.065664
& 0.0345874
& 0.014448
& 0.00708024
& 0.00288559
  \\ \cline{2-7} 
\multicolumn{1}{||l|}{}                  & $R^0$ & - & 1.91
& 1.93
& 1.90
& 2.01
 \\ \cline{2-7} 
\multicolumn{1}{||l|}{}                  & $E^1_h$ &  0.610072
& 0.444263
& 0.287752
& 0.204284
& 0.128841
 \\ \cline{2-7} 
\multicolumn{1}{||l|}{}                  & $R^1$ & - & 0.94
& 0.96
& 0.91
& 1.03
  \\ \cline{2-7} 
\multicolumn{1}{||l|}{}                  & $E^\infty_h$ & 0.0743104
& 0.0384007
& 0.0158235
& 0.0075526
& 0.00315083
  \\ \cline{2-7} 
\multicolumn{1}{||l|}{}                  & $R^\infty$ & - & 1.96
& 1.96
& 1.97
& 1.96
 \\ 
\hline 
\end{tabular}
\caption{Error $E_h^p$ and rate of convergence $R^p \; p=0,1,\infty$ of the proposed method for Example~\ref{1Dpde2}.}\label{Linftableforpde2freefemf4}
\end{table}

\end{example}

\begin{example}\label{VariableceffPDE} 
Here, we consider time-fractional PDE (\ref{pide}) in two dimensions where $\Omega = (0,1)^2$ the unit square and $T=1$ for $\boldsymbol{x} = (x_1,x_2) \in \Omega $ with  $\textbf{A}(\boldsymbol{x},t)=\begin{bmatrix} 
2-\cos(t) & x_1x_2 \\
x_1x_2 & 2-\sin(t) \\
\end{bmatrix},\;
\textbf{b}(\boldsymbol{x},t) =\begin{bmatrix} 
1 + 2x_1x_2\\
1 + x_1x_2\\
\end{bmatrix}$, $c(\boldsymbol{x},t) =1 -\sin(t)$, and $\lambda =0$. The initial condition $u_0$ and the source term $f$ are chosen according to the exact solution
$ u(\boldsymbol{x},t)= \sin(2\pi x_1)\sin(2\pi x_2)(t^{\alpha}+t^3)$ .
For this example, the computational rate of convergence (ROC) is given in Table~\ref{2DPDE1freefemf4}.

\begin{table}[H]
\centering
\begin{tabular}{||l|l|l|l|l|l|l||}
\hline
\multicolumn{2}{||l|}{$N$} & 4 & 8 & 16 & 32 & 64   \\ \hline \hline
\multicolumn{1}{||l|}{\multirow{6}{*}{$\alpha$=0.2}} & $E^0_h$ & 0.482987
& 0.187402
& 0.0574965
& 0.0185859
& 0.0053143
\\ \cline{2-7} 
\multicolumn{1}{||l|}{}                  & $R^0$ & - & 1.69
& 1.91
& 1.98
& 2.00
\\ \cline{2-7} 
\multicolumn{1}{||l|}{}                  & $E^1_h$ &  6.01222
& 3.7969
& 2.11704
& 1.20824
& 0.648361
\\ \cline{2-7} 
\multicolumn{1}{||l|}{}                  & $R^1$ & - & 0.82
& 0.94
& 0.98
& 0.99
\\ \cline{2-7} 
\multicolumn{1}{||l|}{}                  & $E^\infty_h$ & 1.13525
& 0.457247
& 0.141354
& 0.0462793
& 0.0133405
\\ \cline{2-7} 
\multicolumn{1}{||l|}{}                  & $R^\infty$ & - & 1.63
& 1.90
& 1.96
& 1.99
  \\ \hline \hline
\multicolumn{1}{||l|}{\multirow{6}{*}{$\alpha$=0.5}} & $E^0_h$  & 0.7084
& 0.338061
& 0.145868
& 0.0496857
& 0.0186465
  \\ \cline{2-7} 
\multicolumn{1}{||l|}{}                  & $R^0$ &  - & 1.45
& 1.79
& 1.92
& 1.97
  \\ \cline{2-7} 
\multicolumn{1}{||l|}{}   & $E^1_h$ & 7.17168
& 5.07323
& 3.36094
& 1.96969
& 1.20822
  \\ \cline{2-7} 
\multicolumn{1}{||l|}{}                  & $R^1$ &  - & 0.68
& 0.88
& 0.95
& 0.98
  \\ \cline{2-7} 
\multicolumn{1}{||l|}{}                  & $E^\infty_h$ & 1.67434
& 0.819977
& 0.35781
& 0.123991
& 0.0464852
  \\ \cline{2-7} 
\multicolumn{1}{||l|}{}                  & $R^\infty$ &  - & 1.40
& 1.76
& 1.89
& 1.98
 \\ \hline \hline
\multicolumn{1}{||l|}{\multirow{6}{*}{$\alpha$=0.8}} & $E^0_h$ & 0.707731
& 0.481174
& 0.246009
& 0.145463
& 0.0573017
  \\ \cline{2-7} 
\multicolumn{1}{||l|}{}                  & $R^0$ & - & 1.34
& 1.65
& 1.83
& 1.92
 \\ \cline{2-7} 
\multicolumn{1}{||l|}{}                  & $E^1_h$ &  7.17293
& 6.01461
& 4.35323
& 3.36113
& 2.11709
  \\ \cline{2-7} 
\multicolumn{1}{||l|}{}                  & $R^1$ & - & 0.61
& 0.80
& 0.90
& 0.95
 \\ \cline{2-7} 
\multicolumn{1}{||l|}{}                  & $E^\infty_h$ & 1.67304
& 1.13317
& 0.597714
& 0.356602
& 0.140779
  \\ \cline{2-7} 
\multicolumn{1}{||l|}{}                  & $R^\infty$ & - & 1.35
& 1.58
& 1.80
& 1.91
   \\ 
\hline  
\end{tabular}
\caption{Error $E_h^p$ and rate of convergence $R^p, \; p=0,1,\infty,$ of the proposed method for Example~\ref{VariableceffPDE}.}\label{2DPDE1freefemf4}
\end{table}

\end{example}

\begin{example}\label{VariableceffPDEnewtime} Consider a two-dimensional time-fractional PDE (\ref{pide}) for $\Omega = (-1,1)^2$ with variable coefficients $\textbf{A}(\boldsymbol{x},t) =\begin{bmatrix} 
4 & x_1x_2t \\
x_1x_2t & 4 \\
\end{bmatrix},\;
\textbf{b}(\boldsymbol{x},t) =\begin{bmatrix} 
x_1 t \\
x_2 t\\
\end{bmatrix},\; c(\boldsymbol{x},t) = x_1 x_2 t^2, \;\lambda = 0\;$ $\forall \boldsymbol{x}=(x_1,\;x_2)\in\Omega,\;t\in(0,1]$. Let $u(\boldsymbol{x},t)= x_1(1-|x_1|)x_2(1-|x_2|)(1 + t^{\alpha})$ be the exact solution. The initial condition $u_0\notin \dot{H}^3$ and the source term $f$ are chosen accordingly. The computational results shown in Table~\ref{2dpdenew2freefemtimef4} are compatible with the theoretical findings.

\begin{table}[H]
\centering
\begin{tabular}{||l|l|l|l|l|l|l||}
\hline
\multicolumn{2}{||l|}{$N$} & 4 & 8 & 16 & 32 & 64  \\ \hline \hline
\multicolumn{1}{||l|}{\multirow{6}{*}{$\alpha$=0.2}} & $E^0_h$ & 0.02912
& 0.00922254
& 0.00256974
& 0.000762915
& 0.000224849
 \\ \cline{2-7} 
\multicolumn{1}{||l|}{}                  & $R^0$ & - & 1.86
& 1.95
& 1.99
& 1.99
 \\ \cline{2-7} 
\multicolumn{1}{||l|}{}                  & $E^1_h$ &  0.260953
& 0.146321
& 0.0772059
& 0.0422931
& 0.0228686
  \\ \cline{2-7} 
\multicolumn{1}{||l|}{}                  & $R^1$ & - & 0.93
& 0.98
& 0.99
& 1.00
  \\ \cline{2-7} 
\multicolumn{1}{||l|}{}                  & $E^\infty_h$ & 0.0388986
& 0.0138199
& 0.00410534
& 0.00126702
& 0.000374496
   \\ \cline{2-7} 
\multicolumn{1}{||l|}{}                  & $R^\infty$ & - & 1.67
& 1.86
& 1.93
& 1.99
  \\ \hline \hline
\multicolumn{1}{||l|}{\multirow{6}{*}{$\alpha$=0.5}} & $E^0_h$  & 0.0389191
& 0.0153385
& 0.00618916
& 0.00220748
& 0.000763096
  \\ \cline{2-7} 
\multicolumn{1}{||l|}{}                  & $R^0$ &  - & 1.82
& 1.93
& 1.97
& 1.99
 \\ \cline{2-7} 
\multicolumn{1}{||l|}{}   & $E^1_h$ & 0.30606
& 0.190638
& 0.120702
& 0.0715521
& 0.0422931
 \\ \cline{2-7} 
\multicolumn{1}{||l|}{}                  & $R^1$ &  - & 0.93
& 0.97
& 1.00
& 0.99
 \\ \cline{2-7} 
\multicolumn{1}{||l|}{}                  & $E^\infty_h$ & 0.0508273
& 0.0222693
& 0.00967891
& 0.00354648
& 0.00126859
   \\ \cline{2-7} 
\multicolumn{1}{||l|}{}                  & $R^\infty$ & - & 1.62
& 1.77
& 1.92
& 1.93
 \\ \hline \hline
\multicolumn{1}{||l|}{\multirow{6}{*}{$\alpha$=0.8}} & $E^0_h$ & 0.0513037
& 0.029182
& 0.0127102
& 0.00619403
& 0.0025724
\\ \cline{2-7} 
\multicolumn{1}{||l|}{}                  & $R^0$ & - &  1.68
& 1.84
& 1.92
& 1.97
  \\ \cline{2-7} 
\multicolumn{1}{||l|}{}                  & $E^1_h$ &  0.347801
& 0.260963
& 0.171709
& 0.120703
& 0.077206
 \\ \cline{2-7} 
\multicolumn{1}{||l|}{}                  & $R^1$ & - & 0.85
& 0.93
& 0.94
& 1.00
  \\ \cline{2-7} 
\multicolumn{1}{||l|}{}                  & $E^\infty_h$ & 0.0625247
& 0.0390078
& 0.0185925
& 0.00969024
& 0.00411551
  \\ \cline{2-7} 
\multicolumn{1}{||l|}{}                  & $R^\infty$ & - & 1.40
& 1.64
& 1.74
& 1.92
   \\ 
\hline 
\end{tabular}
\caption{Error $E_h^p$ and rate of convergence $R^p, \; p=0,1,\infty,$ of the proposed method for Example~\ref{VariableceffPDEnewtime}.}\label{2dpdenew2freefemtimef4}
\end{table}
\end{example}


\begin{example}\label{VariableceffPIDEnew2_v2} For $J=(0,1]$ and $\Omega = (-1,1)^2$ the unit square with $\boldsymbol{x}=(x_1,x_2)$ and  $\boldsymbol{y}=(y_1,y_2)$, consider the problem (\ref{pide})
with the coefficients 
$\textbf{A}(\boldsymbol{x},t) =\begin{bmatrix} 
1 & \frac{x_1x_2t}{8} \\
\frac{x_1x_2t}{8} & 1 \\
\end{bmatrix},\;
\textbf{b}(\boldsymbol{x},t) =\begin{bmatrix} 
x_1 t \\
x_2 t\\
\end{bmatrix},\; c(\boldsymbol{x},t) = x_1 x_2 t, \;\lambda = \frac{1}{2},$ and $g(\boldsymbol{x},\boldsymbol{y}) = \frac{1}{2}e^{-\|\boldsymbol{x} - \boldsymbol{y}\|^2}$, initial condition $u_0(\boldsymbol{x})= x_1(1-|x_1|)x_2(1-|x_2|)$, and the source term $f(\boldsymbol{x},t)=e^{-t}\sin(\pi x_1)\sin(\pi x_2)$. Table \ref{2Dpidefreefemf} demonstrates that the computational rate of convergence is consistent with our theoretical findings.
\begin{table}[H]
\centering
\begin{tabular}{||l|l|l|l|l|l|l||}
\hline
\multicolumn{2}{||l|}{$N$} & 4 & 8 & 16 & 32 & 64  \\ \hline \hline
\multicolumn{1}{||l|}{\multirow{6}{*}{$\alpha$=0.2}} & $E^0_h$ & 217.449
& 0.678879
& 0.000826122
& 0.000246648
& 7.27521e-05
 \\ \cline{2-7} 
\multicolumn{1}{||l|}{}                  & $R^0$ & -
& 9.32
& 10.26
& 1.98
& 1.99
\\ \cline{2-7} 
\multicolumn{1}{||l|}{}                  & $E^1_h$ & 545.361
& 1.69333
& 0.0399092
& 0.0216531
& 0.0115618
\\ \cline{2-7} 
\multicolumn{1}{||l|}{}                  & $R^1$ & -
& 9.33
& 5.73
& 1.00
& 1.02
\\ \cline{2-7} 
\multicolumn{1}{||l|}{}                  & $E^\infty_h$ & 204.602
& 0.383045
& 0.00074878
& 0.00022527
& 6.64172e-05
  \\ \cline{2-7} 
\multicolumn{1}{||l|}{}                  & $R^\infty$ & -
& 10.15
& 9.54
& 1.97
& 1.99
  \\ \hline \hline
\multicolumn{1}{||l|}{\multirow{6}{*}{$\alpha$=0.5}} & $E^0_h$  & 0.0123586
& 0.00472535
& 0.00191004
& 0.000681553
& 0.000235916
 \\ \cline{2-7} 
\multicolumn{1}{||l|}{}                  & $R^0$ & -
& 1.88
& 1.93
& 1.97
& 1.99
  \\ \cline{2-7} 
\multicolumn{1}{||l|}{}   & $E^1_h$ & 0.151735
& 0.100101
& 0.06343
& 0.0369484
& 0.0216539
  \\ \cline{2-7} 
\multicolumn{1}{||l|}{}                  & $R^1$ & -
& 0.81
& 0.97
& 1.03
& 1.00
  \\ \cline{2-7} 
\multicolumn{1}{||l|}{}                  & $E^\infty_h$ & 0.00868891
& 0.00349092
& 0.00145498
& 0.000516097
& 0.000180741
  \\ \cline{2-7} 
\multicolumn{1}{||l|}{}                  & $R^\infty$ & -
& 1.79
& 1.86
& 1.98
& 1.97
  \\ \hline \hline
\multicolumn{1}{||l|}{\multirow{6}{*}{$\alpha$=0.8}} & $E^0_h$ & 0.0164057
& 0.00889068
& 0.00384336
& 0.0018742
& 0.000778252
 \\ \cline{2-7} 
\multicolumn{1}{||l|}{}                  & $R^0$ & -
& 1.82
& 1.86
& 1.92
& 1.97
 \\ \cline{2-7} 
\multicolumn{1}{||l|}{}                  & $E^1_h$ 
& 0.159741
& 0.137783
& 0.0906031
& 0.0633357
& 0.0399031
  \\ \cline{2-7} 
\multicolumn{1}{||l|}{}                  & $R^1$ & -
& 0.44
& 0.93
& 0.96
& 1.04
 \\ \cline{2-7} 
\multicolumn{1}{||l|}{}                  & $E^\infty_h$ & 0.0101306
& 0.00591297
& 0.00262284
& 0.00128834
& 0.000534778
 \\ \cline{2-7} 
\multicolumn{1}{||l|}{}                  & $R^\infty$ & -
& 1.60
& 1.80
& 1.90
& 1.97
  \\ 
\hline \hline\multicolumn{1}{||l|}{\multirow{6}{*}{$\alpha$=0.99}} & $E^0_h$ & 0.0162598
& 0.0118974
& 0.0056508
& 0.00331496
& 0.00169116
 \\ \cline{2-7} 
\multicolumn{1}{||l|}{}                  & $R^0$ & -
& 1.71
& 1.84
& 1.85
& 1.93
 \\ \cline{2-7} 
\multicolumn{1}{||l|}{}                  & $E^1_h$ 
& 0.144525
& 0.14955
& 0.107969
& 0.0836143
& 0.0587894
  \\ \cline{2-7} 
\multicolumn{1}{||l|}{}                  & $R^1$ & -
& -0.19
& 0.80
& 0.89
& 1.01
 \\ \cline{2-7} 
\multicolumn{1}{||l|}{}                  & $E^\infty_h$ & 0.00966815
& 0.00715519
& 0.00360673
& 0.00212418
& 0.0010907
 \\ \cline{2-7} 
\multicolumn{1}{||l|}{}                  & $R^\infty$ & -
& 1.65
& 1.69
& 1.84
& 1.91
  \\ 
\hline 
\end{tabular}
\caption{Error $E_h^p$ and rate of convergence $R^p, \; p=0,1,\infty,$ of the proposed method for Example~\ref{VariableceffPIDEnew2_v2}.}\label{2Dpidefreefemf}
\end{table}
\end{example}


\begin{example}\label{Mertons}  
The aim of this example is to verify the performance of the proposed IMEX-L1 method and the impact of time-graded mesh for the case of non-smooth initial data $u_0$. We price the European put option under the one-dimensional time-fractional Merton's jump-diffusion model (see \cite{MR2873249} for $\alpha\to 1^{-}$) where the truncated domain $\Omega = (-X, X)=(-1, 1)$. The problem is to find $u: \Omega \times J \rightarrow \mathbb{R}$ such that
\begin{align*}
\begin{cases}
&\partial^{\alpha}_t u - \frac{\sigma^2}{2}\frac{\partial ^2 u}{\partial x^2} - \left( r - \frac{\sigma^2}{2}- \lambda \kappa \right)\frac{\partial u}{\partial x} + (r+\lambda)u - \lambda \int_{\Omega} u( y,t)\rho(y-x)dy = \lambda R(x,t) \quad \text{in}\quad \Omega\times J,\\
& u(x,t) = u_B(x,t) \quad \text{on} \quad  \partial \Omega \times J,\\
& u(x,0) = u_0(x) \quad \forall x \in \Omega,
\end{cases}
\end{align*}
where $R(x,t)= \int_{\mathbb{R}/ \Omega} u(y,t)\rho(y-x)dy$ and $\kappa = \int_{\mathbb R}(e^x-1)\rho(x)dx$. Let the parameters in the model be
\begin{align*}
 & T = 1.0,\;\lambda = 0.10,\; K=100,\; \sigma = 0.15,\\
 &r=0.05,\; \sigma_J = 0.45, ~ \mu_J = -0.90, \;S_0 =  K. 
\end{align*}
 The jump density function $\rho(x)$ is given by $
\rho(x) = \displaystyle{\frac{1}{\sqrt{2\pi\sigma^2_J}}e^{-\frac{(x-\mu_J)^2}{2\sigma_J^2}}}$, the boundary condition $u_B(x,t)= \text{max}\{0, S_0e^x - Ke^{-rt}\}$, the initial condition $u_0(x) =\max\{0, S_0e^x - K\}$ is the payoff function,  and $R(x,t) = S_0e^{x+\mu_J+\frac{\sigma_J^2}{2}}\Phi\left(\frac{x-X+\mu_J +\sigma_J^2}{\sigma_J}\right) - Ke^{-rt}\Phi\left(\frac{x-X+\mu_J}{\sigma_J}\right) $, where  $\Phi(y)=\frac{1}{2\pi}\displaystyle{\int_{-\infty}^ye^{-\frac{x^2}{2}}dx}$ is the cumulative distribution of the standard normal density function. 

In this example, the error $E_{h}^m,\;m=0,1,$ has been calculated by using the formula
\begin{align}
    \nonumber & E_{h}^m:= 
        \max\limits_{1 \leq n \leq N}t^{\frac{\alpha}{2-m}}\|u_{h}^n - \widetilde{u}_h(t_n)\|_m.
\end{align}
Numerical results are shown in Table \ref{L2tableformertonf4}. It is observed that the computational rate of convergence  with respect to the above weighted norm is optimal even when the initial condition $u_0$ is not in $H_0^1(\Omega)\cap H^2(\Omega)$.

\begin{table}[H]
\centering
\begin{tabular}{||l|l|l|l|l|l|l||}
\hline
\multicolumn{2}{||l|}{$N$} & 4 & 8 & 16 & 32 & 64  \\ \hline \hline
\multicolumn{1}{||l|}{\multirow{6}{*}{$\alpha$=0.2}} & $E^0_h$ & 1.05581
& 0.00646041
& 0.00106787
& 0.000265719
& 7.59617e-05
 \\ \cline{2-7} 
\multicolumn{1}{||l|}{}                  & $R^0$ & -
& 8.23
& 2.75
& 2.28
& 2.04
 \\ \cline{2-7} 
\multicolumn{1}{||l|}{}                  & $E^1_h$ & 3.72878
& 0.107098
& 0.054879
& 0.02973
& 0.0161072
 \\ \cline{2-7} 
\multicolumn{1}{||l|}{}                  & $R^1$ & -
& 5.73
& 1.02
& 1.01
& 1.00
\\
\cline{2-7}
\multicolumn{1}{||l|}{}   & $E^\infty_h$ & 2.36107
& 0.0106703
& 0.00208492
& 0.000600877
& 0.000213871
   \\ \cline{2-7} 
\multicolumn{1}{||l|}{}  & $R^\infty$ & -
& 8.72
& 2.50
& 2.04
& 1.68
\\ 
\hline \hline
\multicolumn{1}{||l|}{\multirow{6}{*}{$\alpha$=0.5}} & $E^0_h$  & 0.0136039
& 0.0048126
& 0.00188982
& 0.000666165
& 0.000228864
  \\ \cline{2-7} 
\multicolumn{1}{||l|}{}                  & $R^0$ & -
& 2.03
& 1.99
& 1.99
& 2.01
 \\ \cline{2-7} 
\multicolumn{1}{||l|}{}   & $E^1_h$ & 0.218177
& 0.133309
& 0.0838583
& 0.0500168
& 0.0292883
 \\ \cline{2-7} 
\multicolumn{1}{||l|}{}                  & $R^1$ & -
& 0.96
& 0.99
& 0.99
& 1.00
\\
\cline{2-7}
\multicolumn{1}{||l|}{}   & $E^\infty_h$ & 0.0259708
& 0.0101445
& 0.0043731
& 0.00167432
& 0.000613846
   \\ \cline{2-7} 
\multicolumn{1}{||l|}{}  & $R^\infty$ & -
& 1.84
& 1.79
& 1.83
& 1.88
\\ \hline
\hline 
\multicolumn{1}{||l|}{\multirow{6}{*}{$\alpha$=0.8}} & $E^0_h$ & 0.0216922
& 0.0107296
& 0.00416085
& 0.00195343
& 0.000791354
 \\ \cline{2-7} 
\multicolumn{1}{||l|}{}  & $R^0$ & - &
2.09
& 2.10
& 2.02
& 2.02
  \\ \cline{2-7} 
\multicolumn{1}{||l|}{}   & $E^1_h$ & 0.276241
& 0.193256
& 0.121736
& 0.0831036
& 0.0532743
   \\ \cline{2-7} 
\multicolumn{1}{||l|}{}  & $R^1$ & -
& 1.06
& 1.02
& 1.02
& 1.00
\\
\cline{2-7}
\multicolumn{1}{||l|}{}   & $E^\infty_h$ & 0.0347808
& 0.0192111
& 0.00841226
& 0.00430264
& 0.00189286
   \\ \cline{2-7} 
\multicolumn{1}{||l|}{}  & $R^\infty$ & -
& 1.76
& 1.83
& 1.79
& 1.84
\\ \hline \hline
\multicolumn{1}{||l|}{\multirow{6}{*}{$\alpha$=0.99}} & $E^0_h$ & 0.0217925
& 0.0148037
& 0.00655057
& 0.00369525
& 0.00183329
 \\ \cline{2-7} 
\multicolumn{1}{||l|}{}  & $R^0$ & - &
2.12
& 2.01
& 1.99
& 2.01
  \\ \cline{2-7} 
\multicolumn{1}{||l|}{}   & $E^1_h$ & 0.270686
& 0.22184
& 0.147903
& 0.110727
& 0.0781255
   \\ \cline{2-7} 
\multicolumn{1}{||l|}{}  & $R^1$ & -
& 1.09
& 1.00
& 1.01
& 1.00
\\
\cline{2-7}
\multicolumn{1}{||l|}{}   & $E^\infty_h$ & 0.0348849
& 0.0254806
& 0.0117455
& 0.00700172
& 0.00368295
   \\ \cline{2-7} 
\multicolumn{1}{||l|}{}  & $R^\infty$ & -
& 1.72
& 1.91
& 1.80
& 1.84
\\ \hline
\end{tabular}
\caption{Error $E_h^p$ and rate of convergence $R^p, \; p=0,1,$ of the proposed method for one dimension Merton's jump-diffusion model in Example \ref{Mertons}. }\label{L2tableformertonf4}
\end{table}

\end{example}

\begin{remark}
    For $\alpha=0.2$, the grading parameter $\gamma = \frac{2(2-\alpha)}{\alpha}=18$, the maximum time step size is not small enough compared to T/2,  which can lead to large errors for $N =4$.
\end{remark}
%
\section{Conclusion}\label{section8}
%
A non-uniform implicit-explicit L1 finite element method (IMEX-L1-FEM) for a class of time-fractional partial differential/integro-differential equations is proposed and analyzed. To derive the stability and convergence estimates of the proposed method for the considered problem, we have proposed a modified discrete fractional Gr\"{o}nwall inequality. Up to a factor of $\log_e(N)$, optimal error estimates with respect to $L^{2}$- and $H^{1}$-norms are derived for the problem with initial data $u_0 \in H_0^1(\Omega)\cap H^2(\Omega)$. When the elliptic operator is self-adjoint, an $L^\infty$ error estimate is obtained for 2D problems. All the estimates derived in this article remains valid when $\alpha\to 1^{-}$. Furthermore, some numerical experiments are conducted, and the outcomes of these numerical experiments confirm our theoretical findings.

\subsection*{Acknowledgments} 
The first author gratefully acknowledges the support provided by the Indian Institute of Technology Goa, India. The second author acknowledges the support provided by the Indian Institute of Technology Goa, India, under the start-up grant project no. 2019/SG/LT/031. Prof. Olivier Pironneau is gratefully acknowledged for his valuable suggestions on the implementation of integro-differential equations in FreeFem++.

\bibliographystyle{plain}
\bibliography{references}

\appendix
\section{Proof of Theorem~\ref{Regularity condition} (Regularity Results)}\label{appendix_sec_reg_result_a}
Let $t_\ast\in(0, T]$ be an arbitrary point. Then for this $t_\ast$, the variational problem (\ref{variation1}) can be rewritten as-
\begin{align}
\label{variation1_reg}
&\begin{cases}
& u(0) = u_0 ,\\
& \partial_t^{\alpha} u(t) + A_0(t_\ast)u(t) = \left(A_0(t_\ast)-A_0(t)\right)u(t) + \left(\lambda \mathcal{I} - A_1(t)\right)u(t) +  f(t)  \quad \text{in }\;H^{-1}(\Omega) \quad a.e. \quad t \in J,
\end{cases}
\end{align}
where the operators $A_0(t)$ and $A_1(t)$, $t\in J$ are defined as 
\begin{align*}
    &  (A_0(t)\phi)(\boldsymbol{x}):=  -\nabla\cdot\boldsymbol{A}(\boldsymbol{x},t)\nabla \phi(\boldsymbol{x}) ,\quad \text{and}\quad(A_1(t)\psi)(\boldsymbol{x}) := \boldsymbol{b}(\boldsymbol{x},t)\cdot\nabla \psi(\boldsymbol{x}) + c(\boldsymbol{x},t)\psi(\boldsymbol{x}) \quad\forall \boldsymbol{x}\in\Omega.
\end{align*}
The solution of the above problem (\ref{variation1_reg}) can be represented by (see, \cite{MR4290515})
\begin{align}\label{solution_variation1}
u(t) &= F_\ast(t)u_0 + \int_0^t E_\ast(t-s)\big(\left(A_0(t_\ast)- A_0(s) + \lambda \mathcal{I} - A_1(s) \right)u(s) + f(s) \big)\;ds ,\quad t\in(0,T],
\end{align}
where
\begin{align*}
    F_\ast(t):= \frac{1}{2\pi i}\int_{\Gamma_{\theta, \delta}}e^{zt}z^{\alpha-1}(z^\alpha + A_0 (t_\ast))^{-1}dz \quad \text{and } \quad  E_\ast(t):= \frac{1}{2\pi i}\int_{\Gamma_{\theta, \delta}}e^{zt}(z^\alpha + A_0 (t_\ast))^{-1}dz
\end{align*}
with integral over a contour (oriented with an increasing imaginary part)
$$\Gamma_{\theta, \delta}:=\{z\in\mathbb{C}: |z|=\delta, |\arg{z}|\leq \theta   \} \cup \{z\in \mathbb{C}: z=\rho e^{\pm i\theta}, \rho \geq \delta, i:=\sqrt{-1} \}, \;\theta\in (\pi/2, \pi).$$
For any $s\geq 0$, the Hilbert space $\dot{H}^s(\Omega)$ equipped with the induced norm $$\|v\|_{\dot{H}^s(\Omega)} = \|(-\Delta)^{s/2} v\| = \sqrt{\sum_{j=1}^{\infty} \lambda_j^s (v,\; \phi_j)^2}$$ is defined by (see, \cite{MR2249024}, \cite{MR4290515}, and \cite{MR4125980})
\begin{align*}
    &\dot{H}^s
(\Omega) = \left\{v \in L^2(\Omega) : \sum_{j=1}^{\infty} \lambda_j^s (v,\; \phi_j)^2\;<\;\infty\right\},
\end{align*}
where $\{(\lambda_j,\;\phi_j)\}_{j=1}^{\infty}$ are the eigenpairs of the eigenvalue problem $-\Delta w = \lambda w\;\text{in}\;\Omega$, $w=0\;\text{on}\;\partial \Omega$, with multiplicity counted, and $\{\phi_j\}_{j=1}^{\infty}$ is an orthonormal basis for $L^2(\Omega)$. In particular, $\dot{H}^0
(\Omega)=L^2(\Omega)$, $\dot{H}^1
(\Omega)=H_0^1(\Omega)$, and $\dot{H}^2
(\Omega)=H_0^1(\Omega)\cap H^2(\Omega)$.
\begin{lemma}\label{propertiesofFandE}
The operators $E_\ast(t)$, $W_k(t):=t^kE_\ast(t)$, $F_\ast(t)$, and $A_0(t)$, $t\in J$, satisfy the following properties:
\begin{enumerate}
\item[(i)] $t^{k}\|A_0^{\theta}(t)E^{(k)}_\ast(t)\| + \|A_0^{\theta}(t)W_k^{(k)}(t)\|   \leq ct^{(1-\theta)\alpha -1} \quad \forall\theta\in[0,1]$, $k=0,1,2,\ldots,K$.
\item[(ii)] $\|A_0^{\theta}(t)W_k^{(k-1)}(t)\|   \leq ct^{(1-\theta)\alpha} \quad \forall\theta\in[0,1]$, $k=0,1,2,\ldots,K$.
\item[(iii)] $I-F_\ast(t)=\int_0^t A_0(t_\ast)E_\ast (s)ds,\;$ $\;t^{k}\|A_0^{\theta}(t)F^{(k)}_\ast(t)\|  \leq ct^{-\theta\alpha}\;$  and $\;t^{k+1}\|A_0^{-\theta}(t)F^{(k+1)}_\ast(t)\|  \leq ct^{\theta\alpha} \\ \forall\theta\in[0,1]$, $k=0,1,2,\ldots,K$.
\item[(iv)] Let $\left\|\boldsymbol{A}(\cdot,t)\right\|_{W^{1,\infty}(\Omega,\mathbb{R}^{d\times d})} + \left\|\frac{d}{dt}\boldsymbol{A}(\cdot,t)\right\|_{W^{1,\infty}(\Omega,\mathbb{R}^{d\times d})}\;\leq\; D\;  \forall t\in J.$ Then $$\|(A_0(t)-A_0(s))v\|\leq c|t-s| \|v\|_{\dot{H}^2(\Omega)}\;\forall v\in\dot{H}^2(\Omega).$$
\item[(v)] $c_1 \|\phi\|_{\dot{H}^{2\theta}} \leq \|A_0^\theta(t)\phi\| \leq c_2 \|\phi\|_{\dot{H}^{2\theta}} \;\forall \phi\in \dot{H}^{2\theta},\; \theta \in [0,1]$,
\end{enumerate}
for some positive constants $c$, $D$, $c_1$, and $c_2$ which are independent of $\theta$ and $t$, where  $\psi^{(-1)}(t):=\int_0^t \psi(s)ds$, $\displaystyle \psi^{(k)}(t):=\frac{d^k}{dt^k}\psi(t)$, and the constant $c$ may depend on the positive integer $K$.
\end{lemma}

\begin{proof}
    An appeal to Theorem~6.4 in \cite{MR4290515} (see, page 190), Lemma~6.2 in \cite{MR4290515} (see, page~189), Lemma~6.5 in \cite{MR4290515} (see, page~211), and interpolation yields $(i)$, $(ii)$, $(iii)$, and $(iv)$. $(v)$ follows from the definition of $A_0(t)$ (see, \cite{MR2249024}).
\end{proof}
\begin{lemma}{[\cite{MR4290515}, Lemma~6.7]}\label{recurrencerelations_of_f}
For $f\in W^{k+1,1}(J;L^2(\Omega)), \;k=0,1,2,\ldots,K$, and $\theta \in [0,1]$, there exists a positive constant $C_K$ such that
$$\left\|A_0^\theta(t_\ast)\frac{d^k}{dt^k}\left(t^k\int_0^t E_\ast(t-s)f(s)\;ds\right)\Big|_{t=t_\ast}\right\|\leq C_K t_\ast^{(1-\theta)\alpha}\|f\|_{W^{k+1,1}(J;L^2(\Omega))}\quad \forall k = 0,1,2,\ldots,K.$$
\end{lemma}
\begin{lemma}{[\cite{MR4290515}, Theorem~4.2]}
    \label{fractionalgronwall}
    Let $\beta>0$ and $a(t)=at^{-\alpha}\in L^1_+((0,T])$ with $\alpha \in (0,1),\;a>0$ and $0 \leq b =b(t) \in C([0,T])$ be non-decreasing. Let $u(t)\in L^1_+((0,T]))$ satisfy
    \begin{align*}
        u(t) \leq a(t) + \frac{b}{\Gamma(\beta)}\int_0^t(t-s)^{\beta-1}u(s)\;ds, \quad t \in (0,T].
    \end{align*}
    Then, there holds
    \begin{align*}
     u(t) \leq a \Gamma(1-\alpha)E_{\beta,1-\alpha}(bt^\beta)t^{-\alpha} \quad \text{on} \; (0,T]. 
    \end{align*}
\end{lemma}

\begin{theorem}\label{reg_results_a1} Let $f \in W^{k+1,1}(J;L^2(\Omega)),\;k=0,1,2,\ldots, K$, $\|\mathcal{I}\phi\|_s\;\leq\;C_0\|\phi\|_s\; \forall \phi\in \dot{H}^s(\Omega),\;0\leq s \leq 1,$ and for $k=0,1,2,\ldots,K+1,$
\begin{align*}
        & \left\|\frac{d^{k}}{dt^{k}}\boldsymbol{A}(\cdot,t)\right\|_{W^{1,\infty}(\Omega,\mathbb{R}^{d\times d})}+\left\|\frac{d^k}{dt^k}\boldsymbol{b}(\cdot,t)\right\|_{L^{\infty}(\Omega,\mathbb{R}^{d})} +\left\|\frac{d^k}{dt^k}c(\cdot,t)\right\|_{L^{\infty}(\Omega)}\;\leq\; C_1 \quad  \forall t\in J,
\end{align*}
for some positive constants $C_0$ and $C_1$. Then there exists a positive constant $C=C(C_0,C_1, T, K, \alpha)$ depending on $C_0$, $C_1$, $T$, $K$, and $\alpha$ such that
\begin{align*}
    &\|t^k u^{(k)}(t)\|_{\dot{H}^{2}(\Omega)}\leq C \left(\|u_0\|_{\dot{H}^{2}(\Omega)}+\|f\|_{W^{k+1,1}(J;L^2(\Omega))}\right),\;k=0,1,2,\ldots,K.
\end{align*}
\end{theorem}
\begin{proof}
It is enough to prove the following result, for $n=0,1,2,\ldots, K$,
\begin{align}\label{inductionresult}
    &\|(t^n u(t))^{(n)}\|_{\dot{H}^{2}(\Omega)}\leq C \left(\|u_0\|_{\dot{H}^{2}(\Omega)}+\|f\|_{W^{n+1,1}(J;L^2(\Omega))}\right), \quad t\in J.
\end{align}
We derive this result (\ref{inductionresult}) via mathematical induction. For $n=0 $, and $t_\ast \in(0,T],$ operate $A_0(t_\ast)$ on both the sides of (\ref{solution_variation1}) to arrive at
\begin{align*}
A_0(t_\ast)u(t) =&\; A_0(t_\ast)F_\ast(t)u_0 + \int_0^t A_0(t_\ast)E_\ast(t-s) f(s)\;ds+ \int_0^t A_0(t_\ast)E_\ast(t-s)\left(A_0(t_\ast)- A_0(s) \right)u(s)\;ds\\
&\;+ \int_0^t A_0(t_\ast)E_\ast(t-s)\left(\lambda \mathcal{I}-A_1(s) \right)u(s)\;ds.
\end{align*}
Now, an appeal to Lemma~\ref{propertiesofFandE} $(i)$, $(iii)$, $(iv)$ and $(v)$ yields the following estimate
\begin{align*}
    \|u(t)\|_{\dot{H}^{2}(\Omega)}\leq &\; c_2\Big(\|A_0(t_\ast)F_\ast(t)u_0\| + \left\|\int_0^t A_0(t_\ast)E_\ast(t-s) f(s)\;ds\right\|\\
&+ \int_0^t \|A_0(t_\ast)E_\ast(t-s)\left(A_0(t_\ast)- A_0(s) \right)u(s)\|ds\\
&+ \int_0^t \|A_0^{1-\frac{\epsilon}{2}}(t_\ast)E_\ast(t-s)A_0^{\frac{\epsilon}{2}}(t_\ast)\left(\lambda \mathcal{I}-A_1(s) \right)u(s)\|ds\Big),\;0<\epsilon <0.5\\
   \leq &\;C\Big( \|u_0\|_{\dot{H}^{2}(\Omega)}+  \left\|\int_0^t A_0(t_\ast)E_\ast(t-s) f(s)\;ds\right\| + \int_0^t(t-s)^{-1}(t_\ast -s)\|u(s)\|_{\dot{H}^{2}(\Omega)}ds\\
    &+ \int_0^t(t-s)^{\epsilon\frac{\alpha}{2}-1}\|\left(\lambda \mathcal{I}-A_1(s) \right)u(s)\|_{H^{\epsilon}(\Omega)}ds\Big),\;0<\epsilon <0.5,
\end{align*}
and then taking $t=t_\ast$, $\epsilon = 1/4$, applying  $\|\left(\lambda \mathcal{I}-A_1(s) \right)v\|_{H^{1/4}(\Omega)}\leq C\|v\|_{H^{2}(\Omega)}$ and Lemma~\ref{recurrencerelations_of_f}, we obtain
\begin{align*}
    \|u(t_\ast)\|_{\dot{H}^{2}(\Omega)}
   \leq & \;C\Big( \|u_0\|_{\dot{H}^{2}(\Omega)}+\|f\|_{W^{1,1}(J;L^2(\Omega))} + \int_0^{t_\ast}\|u(s)\|_{\dot{H}^{2}(\Omega)}ds + \int_0^{t_\ast}({t_\ast}-s)^{\frac{\alpha}{8}-1}\|u(s)\|_{\dot{H}^{2}(\Omega)}ds\Big).
\end{align*}
Thus, an appeal to the Lemma~\ref{fractionalgronwall} (Gr\"{o}nwall inequality) yields
\begin{align*}
    \|u(t_\ast)\|_{\dot{H}^{2}(\Omega)}
   \leq &\;C\Big( \|u_0\|_{\dot{H}^{2}(\Omega)}+\|f\|_{W^{1,1}(J;L^2(\Omega))} \Big).
\end{align*}
As $t_\ast\in(0, T]$ is arbitrary, it follows the result for $k=0$. Now, assume that the result (\ref{inductionresult}) holds for $n=0,1,\ldots,k-1 <K$. After multiplying by $t^k$ on both the sides of (\ref{solution_variation1}), consider the $k^{th}$ derivative of the resulting equation
 \begin{align*}
    \frac{d^k}{dt^k}(t^ku(t))=\;& \frac{d^k}{dt^k}(t^kF_\ast(t))u_0 + \frac{d^k}{dt^k}\left(t^k\int_0^t E_\ast(t-s)f(s)\;ds\right) \\
    &+ \frac{d^k}{dt^k} \left(t^k\int_0^t E_\ast(t-s)\big(A_0(t_\ast)- A_0(s) + \lambda \mathcal{I}-A_1(s)\big)u(s)\;ds\right).
\end{align*}
Apply $\displaystyle t^k=\sum_{m=0}^k\binom{k}{m} (t-s)^m s^{k-m},\;0<s\leq t$, and then use the change of variables with product rule of differentiation to find that
\begin{align}\label{A4ideqn}
    \nonumber\frac{d^k}{dt^k}(t^ku(t))=\;& \frac{d^k}{dt^k}(t^kF_\ast(t))u_0 + \frac{d^k}{dt^k}\left(t^k\int_0^t E_\ast(t-s)f(s)\;ds\right) \\
    &+\sum_{m=0}^k\binom{k}{m}\sum_{n=0}^{k-m}\binom{k-m}{n}\int_0^t W_m^{(m)}(t-s)\big(A_0(t_\ast)- A_0(s) + \lambda \mathcal{I}-A_1(s)\big)^{(n)}v_{k-m}^{(k-m-n)}(s)\;ds,  
\end{align}
where $W_m(t) := t^mE_\ast(t)$ and $v_m(t):= t^mu(t)$, $t\in(0,T]$.
Operate $A_0(t_\ast)$ on both the sides of equation (\ref{A4ideqn}), and then apply Lemma~\ref{propertiesofFandE} $(iii)$ and $(v)$ to obtain
\begin{align}\label{A5ideqn}
    \nonumber\left\|\frac{d^k}{dt^k}(t^ku(t))\right\|_{\dot{H}^2(\Omega)}\leq\;& c\|u_0\|_{\dot{H}^2(\Omega)} + \left\|A_0(t_\ast)\frac{d^k}{dt^k}\left(t^k\int_0^t E_\ast(t-s)f(s)\;ds\right)\right\| \\
    &+\sum_{m=0}^k\binom{k}{m}\sum_{n=0}^{k-m}\binom{k-m}{n}\left\|I^k_{m,n}(t)\right\|,  
\end{align}
where $I^k_{m,n}(t):=\int_0^t A_0(t_\ast)W_m^{(m)}(t-s)\big(A_0(t_\ast)- A_0(s) + \lambda \mathcal{I}-A_1(s)\big)^{(n)}v_{k-m}^{(k-m-n)}(s)\;ds$, $m=0,1,\ldots,k$, $n=0,1,\ldots, k-m$. Now, we use Lemma~\ref{propertiesofFandE} $(i)$, $(iv)$ and $(v)$  to estimate $I^k_{m,n}(t_\ast)$. For $I_{0,0}^k (t)$
\begin{align*}
    \left\|I^k_{0,0}(t)\right\|\leq & \int_0^t\left\|A_0(t_\ast)E_\ast(t-s)\big(A_0(t_\ast)- A_0(s) + \lambda \mathcal{I}-A_1(s)\big)v_{k}^{(k)}(s)\right\|ds\\
    \leq & \int_0^t\|A_0(t_\ast)E_\ast(t-s)\big(A_0(t_\ast)- A_0(s)\big)v_{k}^{(k)}(s)\|ds\\
    &+  \int_0^t \|A_0^{1-\frac{\epsilon}{2}}(t_\ast)E_\ast(t-s)A_0^{\frac{\epsilon}{2}}(t_\ast)\left(\lambda \mathcal{I}-A_1(s) \right)v_{k}^{(k)}(s)\|ds,\;0<\epsilon <0.5\\
    \leq& \;C\left(\int_0^t(t-s)^{-1}(t_\ast -s)\|v_{k}^{(k)}(s)\|_{\dot{H}^{2}(\Omega)}ds+ \int_0^t(t-s)^{\epsilon\frac{\alpha}{2}-1}\|v_{k}^{(k)}(s)\|_{\dot{H}^{1+\epsilon}(\Omega)}ds\right),\;0<\epsilon <0.5.
\end{align*}
At $t=t_\ast$ use $\epsilon=1/4$, to find that
\begin{align}\label{A6ideqn}
    \left\|I^k_{0,0}(t_\ast)\right\|\leq & \;C\left(\int_0^{t_\ast}\|v_{k}^{(k)}(s)\|_{\dot{H}^{2}(\Omega)}ds+ \int_0^{t_\ast}(t_\ast-s)^{\frac{\alpha}{8}-1}\|v_{k}^{(k)}(s)\|_{\dot{H}^{2}(\Omega)}ds\right).
\end{align}
For $m=1,2,\ldots,k$, apply Lemma~\ref{propertiesofFandE} $(i)$, $(iii)$, $(iv)$ and $(v)$ 
 and repeat the previous argument to arrive at
\begin{align*}
    &\left\|I^k_{m,0}(t)\right\|\leq  \int_0^t\left\|A_0(t_\ast)W_m^{(m)}(t-s)\big(A_0(t_\ast)- A_0(s) + \lambda \mathcal{I}-A_1(s)\big)v_{k-m}^{(k-m)}(s)\right\|ds\\
    &\leq C\left(\int_0^t(t-s)^{-1}(t_\ast -s)\|v_{k-m}^{(k-m)}(s)\|_{\dot{H}^{2}(\Omega)}ds+ \int_0^t(t-s)^{\frac{\alpha}{8}-1}\|v_{k-m}^{(k-m)}(s)\|_{\dot{H}^{5/4}(\Omega)}ds\right)\\
    &\leq C\left(\int_0^t(t-s)^{-1}(t_\ast -s)\|(s^{k-m}u(s))^{(k-m)}(s)\|_{\dot{H}^{2}(\Omega)}ds+ \int_0^t(t-s)^{\frac{\alpha}{8}-1}\|(s^{k-m}u(s))^{(k-m)}(s)\|_{\dot{H}^{2}(\Omega)}ds\right).
\end{align*}
Thus, at $t=t_\ast$ an induction hypothesis yields
\begin{align}\label{A7ideqn}
    \left\|I^k_{m,0}(t_\ast)\right\|\leq & \;C\left(\|u_0\|_{\dot{H}^{2}(\Omega)}+\|f\|_{W^{k-m+1,1}(J;L^2(\Omega))}\right),\;m=1,2,\ldots,k.
\end{align}
For $m=0,n=1$, apply integration by parts, Lemma~\ref{propertiesofFandE} $(iii)$ and $(v)$, induction hypothesis, and product rule, to obtain
\begin{align}\label{A8ideqn}
    \nonumber\left\|I^k_{0,1}(t)\right\|\leq & \int_0^t\left\|A_0(t_\ast)W_0^{(-1)}(t-s)\big(A_0(t_\ast)- A_0(s) + \lambda \mathcal{I}-A_1(s)\big)^{(2)}v_{k}^{(k-1)}(s)\right\|ds\\
    \nonumber& + \int_0^t\left\|A_0(t_\ast)W_0^{(-1)}(t-s)\big(A_0(t_\ast)- A_0(s) + \lambda \mathcal{I}-A_1(s)\big)^{(1)}v_{k}^{(k)}(s)\right\|ds\\
\nonumber\leq & \;C\left(\int_0^t\left\|v_{k}^{(k-1)}(s)\right\|_{\dot{H}^2(\Omega)}ds + \int_0^t\left\|v_{k}^{(k)}(s)\right\|_{\dot{H}^2(\Omega)}ds\right)\\
\leq & \;C\left(\|u_0\|_{\dot{H}^{2}(\Omega)}+\|f\|_{W^{k,1}(J;L^2(\Omega))} + \int_0^t\left\|v_{k}^{(k)}(s)\right\|_{\dot{H}^2(\Omega)}ds\right).
\end{align}
Now, for $m=0,1,\dots,k$, $n=1,2,\ldots,k-m$, and $(m,n)\neq(0,1)$ use of integration by parts with Lemma~\ref{propertiesofFandE} $(i)$, and $(v)$, induction hypothesis, and product rule yields
\begin{align}\label{A9ideqn}
    \nonumber&\left\|I^k_{m,n}(t)\right\|\leq  \;\int_0^t\left\|A_0(t_\ast)W_m^{(m-1)}(t-s)\big(A_0(t_\ast)- A_0(s) + \lambda \mathcal{I}-A_1(s)\big)^{(n+1)}v_{k-m}^{(k-m-n)}(s)\right\|ds\\
    \nonumber&\hspace{1cm} + \int_0^t\left\|A_0(t_\ast)W_m^{(m-1)}(t-s)\big(A_0(t_\ast)- A_0(s) + \lambda \mathcal{I}-A_1(s)\big)^{(n)}v_{k-m}^{(k-m-n+1)}(s)\right\|ds\\
\nonumber &\hspace{.3cm}\leq\; C\left(\int_0^t\left\|v_{k-m}^{(k-m-n)}(s)\right\|_{\dot{H}^2(\Omega)}ds + \int_0^t\left\|v_{k-m}^{(k-m-n+1)}(s)\right\|_{\dot{H}^2(\Omega)}ds\right)\\
 &\hspace{.3cm}\leq\; C\left(\|u_0\|_{\dot{H}^{2}(\Omega)}+\|f\|_{W^{k-m-n+2,1}(J;L^2(\Omega))}\right).
\end{align}
At $t=t_\ast$, by applying Lemma~\ref{recurrencerelations_of_f}, and estimates (\ref{A6ideqn}--\ref{A9ideqn}) in (\ref{A5ideqn}), we obtain
\begin{align}\label{A10ideqn}
    \nonumber\left\|\frac{d^k}{dt^k}(t^ku(t))\Big|_{t=t_\ast}\right\|_{\dot{H}^2(\Omega)}\leq\;& C\Big(\|u_0\|_{\dot{H}^{2}(\Omega)}+\|f\|_{W^{k+1,1}(J;L^2(\Omega))}\\
    &+ \int_0^{t_\ast}\|(s^k u(s))^{(k)}(s)\|_{\dot{H}^{2}(\Omega)}ds+ \int_0^{t_\ast}(t_\ast-s)^{\frac{\alpha}{8}-1}\|(s^k u(s))^{(k)}(s)\|_{\dot{H}^{2}(\Omega)}ds\Big).  
\end{align}
Finally, as $t_\ast\in (0,T]$ is arbitrary, an appeal to the Lemma~\ref{fractionalgronwall} (Gr\"{o}nwall inequality) yields the assertion (\ref{inductionresult}) for $n=k$, and hence the result follows by induction. This completes the proof. 
\end{proof}
\begin{theorem}\label{reg_results_a1_L2} Let $f \in W^{k+1,1}(J;L^2(\Omega)),\;k=0,1,2$, $\|\mathcal{I}\phi\|_s\;\leq\;C_0\|\phi\|_s\; \forall \phi\in \dot{H}^s(\Omega),\;0\leq s \leq 1,$ and for $k=0,1,2,$
\begin{align*}
        & \left\|\frac{d^{k}}{dt^{k}}\boldsymbol{A}(\cdot,t)\right\|_{W^{1,\infty}(\Omega,\mathbb{R}^{d\times d})}+\left\|\frac{d^k}{dt^k}\boldsymbol{b}(\cdot,t)\right\|_{L^{\infty}(\Omega,\mathbb{R}^{d})} +\left\|\frac{d^k}{dt^k}c(\cdot,t)\right\|_{L^{\infty}(\Omega)}\;\leq\; C_1\quad  \forall t\in J,
\end{align*}
for some positive constants $C_0$ and $C_1$. Then there exists a positive constant $C=C(C_0,C_1, T, K, \alpha)$ depending on $C_0$, $C_1$, $T$, $K$, and $\alpha$ such that
\begin{align*}
    &\|(t^k u^{(k)}(t))\|_{L^{2}(\Omega)}\leq C t^{\alpha} \left(\|u_0\|_{\dot{H}^{2}(\Omega)}+\|f\|_{W^{k+1,1}(J;L^2(\Omega))}\right),\;k=1,2.
\end{align*}
\end{theorem}
\begin{proof}
Differentiate (\ref{solution_variation1}) with respect to $t$ to obtain
\begin{align*}
u^\prime(t)&=F_\ast^\prime(t)u_0 + \frac{d}{dt}\int_0^tE_\ast(t-s)(A_0(t_\ast)-A_0(s)+\lambda \mathcal{I}-A_1(s))u(s)\;ds  + \frac{d}{dt}\int_0^tE_\ast(t-s)f(s)\;ds,
\end{align*}
and then apply Lemma~\ref{propertiesofFandE} $(iii)$ and the change of variable with product rule of differentiation to get
\begin{align}\label{differ}
\nonumber&u^\prime(t)
= - E_\ast(t)A_0(0)u_0  +E_\ast(t)(\lambda \mathcal{I}-A_1(0))u_0 + E_\ast(t) f(0) +\int_0^tE_\ast(t-s)f^\prime(s)\;ds\\
&+  \int_0^t E_\ast(t-s)(A_0(s)+A_1(s))^\prime u(s)\;ds - \int_0^t E_\ast(t-s)(A_0(t_\ast)-A_0(s)+\lambda \mathcal{I}- A_1(s))u^\prime(s)\;ds.  
\end{align}
Now, an appeal to Lemma~\ref{propertiesofFandE} $(i)$, $(iv)$ and $(v)$, and $\|u(s)\|_{\dot{H}^{1}(\Omega)}\leq\|u(s)\|_{\dot{H}^{2}(\Omega)}$ yields 
\begin{align*}
&\|u^\prime(t)\| \leq ct^{\alpha-1}\|u_0\|_{\dot{H}^2(\Omega)} +ct^{\alpha-1}\|f(0)\| + \int_0^t\| E_\ast(t-s) f^\prime(s)\|ds+ \int_0^t\|E_\ast(t-s)(A_0(s)-A_1(s))^\prime u(s)\|ds\\
&+\int_0^t\|E_\ast(t-s)(A_0(t_\ast)-A_0(s))u^\prime(s)\|ds+ \int_0^t\|A_0^{\frac{1}{2}}(t_\ast)E_\ast(t-s)A_0^{-\frac{1}{2}}(t_\ast)(\lambda \mathcal{I} - A_1(s))u^\prime(s)\|ds\\
&\leq ct^{\alpha-1}\|u_0\|_{\dot{H}^2(\Omega)} +ct^{\alpha-1}\|f(0)\| + \int_0^t (t-s)^{\alpha-1}\| f^\prime(s)\|ds+ c\int_0^t(t-s)^{\alpha-1}\|u(s)\|_{\dot{H}^2(\Omega)}ds\\
&+ c\int_0^t(t-s)^{-1}(t_\ast-s)\|u^\prime(s)\|ds+ \int_0^t(t-s)^{\frac{\alpha}{2}-1}\|u^\prime(s)\|ds.
\end{align*}
 At $t=t_\ast$, an application of Theorem~\ref{reg_results_a1} with $k=0$ yields
\begin{align*}
\|u^\prime(t_\ast)\| &\leq ct_\ast^{\alpha-1} \Big( \|u_0\|_{\dot{H}^{2}(\Omega)}+\|f\|_{W^{2,1}(J;L^2(\Omega))} \Big) + \int_0^{t_\ast}\|u^\prime(s)\|ds+ \int_0^{t_\ast}(t_\ast-s)^{\frac{\alpha}{2}-1}\|u^\prime(s)\|ds,
\end{align*}
and thus, an application of Lemma~\ref{fractionalgronwall} (Gr\"{o}nwall inequality) yields
\begin{align*}
\|u^\prime(t_\ast)\| &\leq ct_\ast^{\alpha-1} \Big( \|u_0\|_{\dot{H}^{2}(\Omega)}+\|f\|_{W^{2,1}(J;L^2(\Omega))} \Big).
\end{align*}
As $t_\ast\in(0,T]$ was arbitrary, the result 
\begin{align}\label{L2normkis0}
\|tu^\prime(t)\| &\leq ct^{\alpha} \Big( \|u_0\|_{\dot{H}^{2}(\Omega)}+\|f\|_{W^{2,1}(J;L^2(\Omega))} \Big) \;\forall t\in(0,T],
\end{align}
 holds for $k=1$. For $k=2$, we note that 
\begin{align}\label{A4ideqnL2}
   \nonumber&(tu^\prime(t))^{\prime}=\; - (tE_\ast(t) )^{\prime} A_0(0)u_0+(tE_\ast(t) )^{\prime}(\lambda \mathcal{I}-A_1(0))u_0 +(tE_\ast(t) )^{\prime}f(0) +tE_\ast(t) f^{\prime}(0)+\int_0^tE_\ast(t-s)f^{\prime}(s)\;ds\\
   \nonumber&+t\int_0^tE_\ast(t-s)f^{\prime\prime}(s)\;ds+  \frac{d}{dt}\left(t\int_0^t E_\ast(t-s)(A_0^\prime(s)+A_1^\prime(s))u(s)\;ds \right)\\
   &- \frac{d}{dt}\left(t\int_0^t E_\ast(t-s)(A_0(t_\ast)-A_0(s)+\lambda \mathcal{I}- A_1(s))u^\prime(s)\;ds  \right),
   \end{align}
and therefore an appeal to Lemma~\ref{propertiesofFandE} $(i)$ and $(v)$, and Lemma~\ref{recurrencerelations_of_f} yields
\begin{align}\label{L2normKis1}
&\|(tu^\prime(t))^{\prime}\|\leq\;ct^{\alpha-1} \Big( \|u_0\|_{\dot{H}^{2}(\Omega)}+\|f\|_{W^{3,1}(J;L^2(\Omega))} \Big)+\|I_1(t)\|+\|I_2(t)\|
\end{align}
with
\begin{align}\label{I1}
\nonumber I_1(t)=&\;\frac{d}{dt}\left(t\int_0^t E_\ast(t-s)(A_0^\prime(s)+A_1^\prime(s))u(s)ds \right)\\
\nonumber=&\int_0^t E_\ast(t-s)(A_0(s)+A_1(s))^{\prime\prime} su(s)ds+  \int_0^t E_\ast(t-s)(A_0(s)+A_1(s))^{\prime} (su(s))^\prime ds \\
&+\int_0^t ((t-s)E_\ast(t-s))^\prime(A_0(s)+A_1(s))^{\prime} u(s)ds
\end{align}
and
\begin{align}
\nonumber I_2(t)=&\; \frac{d}{dt}\left(t\int_0^t E_\ast(t-s)(A_0(t_\ast)-A_0(s)+\lambda \mathcal{I}- A_1(s))u^\prime(s)ds  \right)\\
\nonumber=& \int_0^t E_\ast(t-s)(A_0(s)+ A_1(s))^\prime su^\prime(s)ds + \int_0^t E_\ast(t-s)(A_0(t_\ast)-A_0(s)+\lambda \mathcal{I}- A_1(s))(su^\prime(s))^\prime ds \\
  \label{I2} &+ \int_0^t ((t-s)E_\ast(t-s))^\prime(A_0(t_\ast)-A_0(s)+\lambda \mathcal{I}- A_1(s))u^\prime(s)ds,
\end{align}
where we have applied $t=(t-s)+s,\;0<s\leq t$, change of variables and product rule of differentiation.

Now, using Lemma~\ref{propertiesofFandE}  $(i)$ and $(v)$, and applying  $\|u(s)\|_{\dot{H}^{1}(\Omega)}\leq\|u(s)\|_{\dot{H}^{2}(\Omega)}$, we obtain the following estimate for $I_1(t)$
\begin{align*}
 \|I_1(t)\|&\leq c\int_0^t (t-s)^{\alpha-1}\|u(s)\|_{\dot{H}^2(\Omega)}ds + c\int_0^t (t-s)^{\alpha-1}\|su^\prime(s)\|_{\dot{H}^2(\Omega)}ds   
\end{align*}
and at $t=t_\ast$ apply Theorem \ref{reg_results_a1} for $k=0$ and $k=1$, to find that
\begin{align}\label{I1estimate}
 \|I_1(t_\ast)\|&\leq ct_\ast^{\alpha-1} \Big( \|u_0\|_{\dot{H}^{2}(\Omega)}+\|f\|_{W^{2,1}(J;L^2(\Omega))} \Big).
\end{align}
To estimate $I_2(t)$,  apply Lemma~\ref{propertiesofFandE}  $(i)$, $(iv)$ and $(v)$, and $\|u(s)\|_{\dot{H}^{1}(\Omega)}\leq\|u(s)\|_{\dot{H}^{2}(\Omega)}$ as follows
\begin{align*}
    &\|I_2(t)\|\leq \int_0^t \|E_\ast(t-s)(A_0(s)+ A_1(s))^\prime su^\prime(s)\|ds + \int_0^t \|E_\ast(t-s)(A_0(t_\ast)-A_0(s))(su^\prime(s))^\prime \|ds \\
   &+ \int_0^t \|A_0^{\frac{1}{2}}(t_\ast)E_\ast(t-s)A_0^{-\frac{1}{2}}(t_\ast)(\lambda \mathcal{I}- A_1(s))(su^\prime(s))^\prime \|ds+ \int_0^t \|((t-s)E_\ast(t-s))^\prime(A_0(t_\ast)-A_0(s))u^\prime(s)\|ds\\
   &+\int_0^t \|A_0^{\frac{1}{2}}(t_\ast)((t-s)E_\ast(t-s))^\prime A_0^{-\frac{1}{2}}(t_\ast)(\lambda \mathcal{I}- A_1(s))u^\prime(s)\|ds\\
   &\leq c\int_0^t (t-s)^{\alpha-1}\|su^\prime(s)\|_{\dot{H}^2(\Omega)}ds  + c\int_0^t(t-s)^{-1}(t_\ast -s)\|(su^\prime(s))^\prime\|ds +c\int_0^t(t-s)^{\frac{\alpha}{2}-1}\|(su^\prime(s))^\prime\|ds \\
   &+c\int_0^t(t-s)^{-1}(t_\ast -s)\|u^\prime(s)\|ds+c\int_0^t(t-s)^{\frac{\alpha}{2}-1}\|u^\prime(s)\|ds.
\end{align*}
At $t=t_\ast$, an application of Theorem~\ref{reg_results_a1} with $k=1$ and (\ref{L2normkis0}) shows
\begin{align}\label{I2estimate}
 \|I_2(t_\ast)\|&\leq ct_\ast^{\alpha-1} \Big( \|u_0\|_{\dot{H}^{2}(\Omega)}+\|f\|_{W^{2,1}(J;L^2(\Omega))} \Big)+ c\int_0^{t_\ast}\|(su^\prime(s))^\prime\|ds +c\int_0^{t_\ast}(t_\ast-s)^{\frac{\alpha}{2}-1}\|(su^\prime(s))^\prime\|ds.
\end{align}
Now, applying Lemma~\ref{recurrencerelations_of_f}, and estimates (\ref{I1estimate}--\ref{I2estimate}) in (\ref{L2normKis1}), we obtain
\begin{align*}
\|(tu^\prime(t))^{\prime}|_{t=t_\ast}\|\leq\;  ct_\ast^{\alpha-1} \Big( \|u_0\|_{\dot{H}^{2}(\Omega)}+\|f\|_{W^{3,1}(J;L^2(\Omega))} \Big)+ c\int_0^{t_\ast}\|(su^\prime(s))^\prime\|ds +c\int_0^{t_\ast}(t_\ast-s)^{\frac{\alpha}{2}-1}\|(su^\prime(s))^\prime\|ds. 
\end{align*}
Finally, as $t_\ast\in (0,T]$ is arbitrary, an appeal to the Lemma~\ref{fractionalgronwall} (Gr\"{o}nwall inequality) yields the result for $k=2$.
\end{proof}
\section{Proof of Theorem~\ref{DFGI} (Discrete fractional Gr\"{o}nwall inequality)}\label{appendix_sec_gronwall_b}
\begin{proof}
Apply the  definition of $D_{t_n}^{\alpha}$ in (\ref{dfgi1}) to arrive at
\begin{align*}
&\displaystyle{\sum_{k=1}^{j}}K^{j,k}_{1-\alpha} \left((v^k)^2-(v^{k-1})^2\right) \leq  \sum_{i=0}^j \lambda^j_{j-i}(v^i)^2 + v^j\xi^j + (\eta^j)^2 + (\zeta^j)^2, \quad 1\leq j \leq N.
\end{align*}  
After multiplying the above inequality by $P^{n,j}_{\alpha}$ and summing the index $j$ from $1$ to $n$, we obtain
\begin{align}
\label{dfgieqn1}
\nonumber\displaystyle{\sum_{j=1}^{n}P^{n,j}_{\alpha}\sum_{k=1}^{j}}K^{j,k}_{1-\alpha} \left((v^k)^2-(v^{k-1})^2\right) &\leq \sum_{j=1}^{n} P^{n,j}_{\alpha} \sum_{i=0}^j \lambda^j_{j-i}(v^i)^2 + \sum_{j=1}^{n} P^{n,j}_{\alpha} v^j\xi^j \\
&+ \sum_{j=1}^{n} P^{n,j}_{\alpha} (\eta^j)^2+ \sum_{j=1}^{n} P^{n,j}_{\alpha} (\zeta^j)^2, \quad 1\leq n \leq N.  
\end{align}
An exchange of order of summation and Lemma~\ref{discretekernelproperties} $(b)$ yields
\begin{align}
\label{dfgieqn2}
\nonumber\displaystyle{\sum_{j=1}^{n}P^{n,j}_{\alpha}\sum_{k=1}^{j}}K^{j,k}_{1-\alpha} \left((v^k)^2-(v^{k-1})^2\right) &= \sum_{k=1}^{n}\left(\sum_{j=k}^n P^{n,j}_{\alpha}K^{j,k}_{1-\alpha} \right)\left((v^k)^2-(v^{k-1})^2\right) \\
&= (v^n)^2 - (v^0)^2 , \quad 1\leq n\leq N.
\end{align}
Further, Lemma~\ref{discretekernelproperties} $(c)$ implies
\begin{align}
\label{dfgieqn3}&\sum_{j=1}^{n} P^{n,j}_{\alpha}  \leq \sum_{j=1}^{n} P^{n,j}_{\alpha} k_1(t_j)  \leq k_{1+\alpha}(t_n) = \frac{t_n^{\alpha}}{\Gamma(1+\alpha)} \leq 2 t_n^{\alpha} , \quad 1\leq n\leq N, 
\end{align}
where we have used $\Gamma(1+\alpha) \geq 2^{\alpha -1} \geq \frac{1}{2} ~ \forall ~\alpha \in [0,1]$. Thus, by using relations (\ref{dfgieqn2}) and (\ref{dfgieqn3}) in (\ref{dfgieqn1}), we obtain
\begin{align}\label{dfgieqn4}
& (v^n)^2  \leq (v^0)^2 + \sum_{j=1}^{n} P^{n,j}_{\alpha} \sum_{i=0}^j \lambda^j_{j-i}(v^i)^2 + \sum_{j=1}^{n} P^{n,j}_{\alpha} v^j \xi^j + 2t_n^\alpha\max_{1\leq j \leq n} (\eta^j)^2+ \sum_{j=1}^{n} P^{n,j}_{\alpha} (\zeta^j)^2, \quad 1\leq n \leq N.
\end{align}  
Now, define a non-decreasing finite sequence $\{\Phi_n \}_{n=1}^{N},$
\begin{align*}
\Phi_n :=  v^0 + \max_{1 \leq j \leq n}\sum_{i=1}^j P^{j,i}_{\alpha}\xi^i + \sqrt{2t_n^{\alpha}}\max_{1 \leq j \leq n} \eta^j + \max_{1 \leq j \leq n} \sqrt{\sum_{i=1}^{j} P^{j,i}_{\alpha} (\zeta^i)^2}, \quad 1\leq n \leq N,
\end{align*}
and derive the required estimate (\ref{dfgi2}), that is,
\begin{align}\label{dfgieqn5}
    v^n \leq 2 E_{\alpha}(2 \Lambda t_n^{\alpha})\;\Phi_n \quad \forall ~ 1 \leq n \leq N
\end{align}
using mathematical induction. 

For $n=1$, if $v^1 < v^0$ or $v^1 < \sqrt{2t_1^\alpha}\;\eta^1$ or $v_1 < \sqrt{P_{\alpha}^{1,1}(\zeta^1)^2}$, then the definition of $\Phi_1$ and $1=E_{\alpha}(0)\leq E_{\alpha}(2\Lambda t_{1}^{\alpha})$ yields the result (\ref{dfgieqn5}). Otherwise, $v^0 \leq v^1$, $\sqrt{2t_1^\alpha}\;\eta^1 \leq v^1$, and $\sqrt{P_{\alpha}^{1,1}(\zeta^1)^2} \leq v_1$ and hence, the inequality (\ref{dfgieqn4}) implies 
\begin{align*}
(v^1)^2
&\leq \left(v^0  +  v^{1}P^{1,1}_{\alpha}\sum_{i=0}^{1}\lambda_{1-i}^{1} +  P^{1,1}_{\alpha} \xi^1  + \sqrt{2 t_1^{\alpha}}\;\eta^1+ \sqrt{P_{\alpha}^{1,1}(\zeta^1)^2} \right)v^1\\
&\leq \left(v^0  +  \frac{1}{2}v^{1} +  P^{1,1}_{\alpha} \xi^1 +\sqrt{2 t_1^{\alpha}}\;\eta^1+ \sqrt{P_{\alpha}^{1,1}(\zeta^1)^2} \right)v^1=\left(\frac{1}{2}v^{1} + \Phi_{1} \right)v^1  \\
&\leq \left(\frac{1}{2}v^{1} + E_{\alpha}(2\Lambda t_{1}^{\alpha})\;\Phi_{1} \right)v^1, 
\end{align*}
where the estimates $ P^{1,1}_{\alpha}\sum_{i=0}^{1}\lambda_{1-i}^{1} \leq \Gamma(2-\alpha)\Delta t_1^{\alpha}\Lambda \leq \frac{1}{2}$ and $1=E_{\alpha}(0)\leq E_{\alpha}(2\Lambda t_{1}^{\alpha})$ have been used. Thus,
\begin{align*}
v^1 & \leq  2E_{\alpha}(2\Lambda t_1^{\alpha})\;\Phi_1
\end{align*}
and therefore the inequality (\ref{dfgieqn5}) holds for $n=1$.

Now, assume that the inequality (\ref{dfgieqn5}) holds for $1 \leq k \leq m-1,\;m\leq N,$ i.e.,
 \begin{align}\label{inducthyp}
 v^k \leq 2  E_{\alpha}(2\Lambda t_k^{\alpha})\;\Phi_k, \quad 1 \leq k \leq m-1.
 \end{align}
Thus, there exists an integer $m_0,\;1 \leq m_0 \leq m-1,$ such that $v^{m_0}= \max_{1 \leq j \leq m-1} v^j$ and corresponding to this index $m_0$ there are two cases.

\noindent \textbf{{Case-I:}} $\boldsymbol{v^m \leq v^{m_0}.}$ In this case, the induction hypothesis (\ref{inducthyp}) along with the non-decreasing properties of $E_{\alpha}$ and $\Phi_n$ yields the result.\\
\noindent \textbf{{Case-II:}} $\boldsymbol{v^m > v^{m_0}.}$ If $v^m < v^0$ or $v^m < \sqrt{2 t_m^{\alpha}}\;\max_{1 \leq j \leq m}\eta^j$ or $v^m < \displaystyle\max_{1 \leq j \leq m}\sqrt{\sum_{i=1}^jP_{\alpha}^{j,i}(\zeta^i)^2}$, then the definition of $\Phi_n$ and $1=E_{\alpha}(0)\leq E_{\alpha}(2\Lambda t_{m}^{\alpha})$ implies that the inequality (\ref{dfgieqn5}) holds for $n=m$. Otherwise, $v^0 \leq v^m$ and $ \sqrt{2 t_m^{\alpha}}\; \max_{1 \leq j \leq m}\eta^j \leq v^m$ and $ \displaystyle\max_{1 \leq j \leq m}\sqrt{\sum_{i=1}^jP_{\alpha}^{j,i}(\zeta^i)^2} \leq v^m$ hence, (\ref{dfgieqn4}) yields
\begin{align*}
(v^m)^2 & \leq \left(v^0 + \sum_{j=1}^{m} P^{m,j}_{\alpha} \sum_{i=0}^j \lambda^j_{j-i}v^i + \sum_{j=1}^{m} P^{m,j}_{\alpha} \xi^j +\sqrt{2t_m^{\alpha}}\;\max_{1\leq j \leq m}\eta^j + \displaystyle\sqrt{\sum_{i=1}^mP_{\alpha}^{m,i}(\zeta^i)^2}\right)v^m\\
&\leq \Bigg(v^0  + \max_{1\leq k \leq m}\sum_{j=1}^{k} P^{k,j}_{\alpha} \xi^j + \sqrt{2t_m^{\alpha}}\;\max_{1\leq j \leq m}\eta^j + \displaystyle\max_{1 \leq k \leq m}\sqrt{\sum_{i=1}^kP_{\alpha}^{k,i}(\zeta^i)^2}  +  \sum_{j=1}^{m-1} P^{m,j}_{\alpha} \sum_{i=0}^j \lambda^j_{j-i}v^i  \\
&+ v^m P^{m,m}_{\alpha} \sum_{i=0}^m \lambda^m_{m-i} \Bigg)v^m\\
&\leq  \left(\Phi_m  +  \sum_{j=1}^{m-1} P^{m,j}_{\alpha} \sum_{i=0}^j \lambda^j_{j-i}v^i  + \frac{1}{2}v^m \right)v^m,
\end{align*}
where $P^{m,m}_{\alpha} \sum_{i=0}^m \lambda^m_{m-i} \leq \Gamma(2-\alpha)\Delta t_m^{\alpha}\Lambda \leq \frac{1}{2}$ is used. Thus,
\begin{align}\label{substitute}
v^m &\leq 2\Phi_m + 2\sum_{j=1}^{m-1} P^{m,j}_{\alpha} \sum_{i=0}^j \lambda^j_{j-i}v^i. 
\end{align}
To estimate the last term in (\ref{substitute}), after applying the induction hypothesis (\ref{inducthyp}), use $v^0 \leq 2 E_{\alpha}(2 \Lambda t _j^{\alpha})\Phi_j,$  $\sum_{i=0}^j \lambda^j_{j-i} \leq \Lambda,$ $\Phi_j\leq \Phi_{j+1} ~ 1 \leq j \leq N-1,$ and Lemma~\ref{discretekernelproperties} $(e)$ to obtain
\begin{align}\label{substitute1}
 \nonumber &2\sum_{j=1}^{m-1} P^{m,j}_{\alpha} \sum_{i=0}^j \lambda^j_{j-i}v^i  \leq    2\sum_{j=1}^{m-1} P^{m,j}_{\alpha}\left( \lambda^j_{j}v^0 + 2 \sum_{i=1}^j \lambda^j_{j-i} E_{\alpha}(2\Lambda t_i^{\alpha})\Phi_i\right)\\
 \nonumber& \leq    4\Lambda \sum_{j=1}^{m-1} P^{m,j}_{\alpha}   E_{\alpha}(2\Lambda t_j^{\alpha})\Phi_j\\  
 &\leq  4\Phi_m\Lambda \sum_{j=1}^{m-1} P^{m,j}_{\alpha}   E_{\alpha}(2\Lambda t_j^{\alpha}) 
 \leq  2\Phi_m E_{\alpha}(2\Lambda t_m^{\alpha}) -2\Phi_m. 
\end{align}
An application of the estimate (\ref{substitute1}) in (\ref{substitute}) yields
\begin{align*}
     & v^m \leq 2  E_{\alpha}(2\Lambda t_m^{\alpha})\;\Phi_m,
\end{align*}
that is, the inequality (\ref{dfgieqn5}) holds for $n=m$. Thus, the mathematical induction confirms the result (\ref{dfgieqn5}) and completes the rest of the proof.
\end{proof}

\end{document}